\documentclass[11pt]{article}
\usepackage{CJK}
\usepackage{CJKnumb}
\usepackage{graphicx}
\usepackage{amsfonts}
\usepackage{amsmath}
\usepackage{amsthm}
\usepackage{amssymb}
\usepackage{fancyhdr}
\usepackage{indentfirst}
\usepackage{titlesec}
\usepackage{hyperref}
\usepackage{abstract}

\newcommand{\al}{\alpha}
\newcommand{\be}{\beta}
\newcommand{\V}{\mathcal{V}}
\newcommand{\R}{\mathbb{R}}
\newcommand{\N}{\mathbb{N}}
\newcommand{\mH}{\mathcal{H}}
\newcommand{\F}{\mathcal{F}}

\newcommand{\B}{\mathcal{B}}

\newcommand{\lan}{\langle}
\newcommand{\ran}{\rangle}
\newcommand{\la}{\lambda}
\newcommand{\lc}{\scalebox{2}{$\llcorner$}}
\newcommand{\La}{\Lambda}

\newcommand{\Si}{\Sigma}
\newcommand{\de}{\delta}
\newcommand{\De}{\Delta}
\newcommand{\ep}{\epsilon}
\newcommand{\pr}{\prime}
\newcommand{\Om}{\Omega}
\newcommand{\Ga}{\Gamma}
\newcommand{\ga}{\gamma}

\newcommand{\lap}{\triangle}
\newcommand{\ti}{\tilde}

\newcommand{\mL}{\mathcal{L}}

\newcommand{\mZ}{\mathbb{Z}}
\newcommand{\Z}{\mathcal{Z}}
\newcommand{\mS}{\mathcal{S}}
\newcommand{\f}{\mathbf{f}}
\newcommand{\m}{\mathbf{m}}
\newcommand{\M}{\mathbf{M}}
\newcommand{\bL}{\mathbf{L}}
\newcommand{\bI}{\mathbf{I}}
\newcommand{\mf}{\mathbf{f}}

\newcommand{\mA}{\mathfrak{A}}
\newcommand{\mF}{\mathbf{F}}
\newcommand{\mR}{\mathcal{R}}

\renewcommand{\baselinestretch}{1.2}
\begin{document}

\newtheorem{theorem}{Theorem}[section]
\newtheorem{proposition}[theorem]{Propostion}
\newtheorem{corollary}[theorem]{Corollary}

\newtheorem{claim}{Claim}

\theoremstyle{remark}
\newtheorem{remark}[theorem]{Remark}

\theoremstyle{definition}
\newtheorem{definition}[theorem]{Definition}

\theoremstyle{plain}
\newtheorem{lemma}[theorem]{Lemma}

\numberwithin{equation}{section}

%
%
%
%
%
%

\pagestyle{headings}
\renewcommand{\headrulewidth}{0.4pt}

\title{\textbf{Min-max minimal hypersurface in $(M^{n+1}, g)$ with $Ric_{g}>0$ and $2\leq n\leq 6$}}
\author{Xin Zhou\\}
\maketitle

\pdfbookmark[0]{}{beg}

\renewcommand{\abstractname}{}    
\renewcommand{\absnamepos}{empty} 
\begin{abstract}
\textbf{Abstract:} In this paper, we study the shape of the min-max minimal hypersurface produced by Almgren-Pitts in \cite{A2}\cite{P} corresponding to the fundamental class of a Riemannian manifold $(M^{n+1}, g)$ of positive Ricci curvature with $2\leq n\leq 6$. We characterize the Morse index, area and multiplicity of this min-max hypersurface. In particular, we show that the min-max hypersurface is either orientable and of index one, or is a double cover of a non-orientable minimal hypersurface with least area among all closed embedded minimal hypersurfaces.
\end{abstract}


\section{Introduction}

Almgren and Pitts developed a min-max theory for constructing embedded minimal hypersurface by global variational method \cite{A1}\cite{A2}\cite{P}. They showed that any Riemannian manifold $(M^{n+1}, g)$ for $2\leq n\leq 5$ has a nontrivial smooth closed embedded minimal hypersurface. Later on, Schoen and Simon \cite{SS} extended to the case of dimension $n=6$\footnote{They also showed the existence of a nontrivial minimal hypersurface with a singular set of Hausdorff dimension $n-7$ when $n\geq 7$.}. In \cite{A1}\cite{A2}, Almgren showed that the fundamental class $[M]\in H_{n+1}(M)$ of an orientable manifold $M$ can be realized as a nontrivial homotopy class in $\pi_{1}\big(\Z_{n}(M), \{0\}\big)$. Here $\Z_{n}(M)$ is the space of integral $n$-cycles in $M$ (see \cite[\S 2.1]{P}). Almgren and Pitts \cite{A2}\cite{P} showed that a min-max construction on the homotopy class in $\pi_{1}\big(\Z_{n}(M), \{0\}\big)$ corresponding to $[M]$ gives a nontrivial smooth embedded minimal hypersurface with possibly multiplicity. We will call this min-max hypersurface the one corresponding to the fundamental class $[M]$. Besides the existence, there is almost no further geometric information known about this min-max minimal hypersurface, e.g. the Morse index\footnote{See \cite[Chap 1.8]{CM2} for the definition of Morse index.}, volume and multiplicity. By the nature of the min-max construction and for the purpose of geometric and topological applications, it has been conjectured and demanded to know that these min-max hypersurfaces should have total Morse index less or equal than one (see \cite{PR}). Recently, Marques and Neves \cite{MN}  gave a partial answer of this question when $n=2$. They showed the existence of an index one heegaard surface in certain three manifolds. 
Later on, in their celebrated proof of the Willmore conjecture \cite{MN2}, Marques and Neves showed that the min-max surface has index five for a five parameter family of sweepouts in the standard three sphere $S^{3}$. 

In this paper, we study the shape of the min-max hypersurface corresponding to the fundamental class $[M]$ in the case when $(M^{n+1}, g)$ has positive Ricci curvature, i.e. $Ric_{g}>0$. In this case there exists no closed embedded stable minimal hypersurface (see \cite[Chap 1.8]{CM2}) in $M$. By exploring this special feature, we will give a characterization of the Morse index, volume and multiplicity of this min-max hypersurface. The study of Morse index was initiated by Marques-Neves in three dimensions \cite{MN}.

We always assume that $(M^{n+1}, g)$ is \emph{connected closed orientable with $2\leq n\leq 6$}. Hypersurfaces $\Si^{n}\subset M^{n+1}$ are always assumed to be \emph{connected closed} and \emph{embedded}. Denote
$$\mS=\{\Si^{n}\subset(M^{n+1}, g):\ \Si^{n} \textrm{ is a minimal hypersurface in } M\}.$$
Hence $\mS\neq\emptyset$ by \cite{P}\cite{SS}\cite{DT}. Let
\begin{equation}\label{minimal volume of embedded minimal hypersurfaces}
\left. W_{M}= \min_{\Si\in\mS}\Big\{ \begin{array}{ll}
V(\Si), \quad \textrm{ if $\Si$ is orientable}\\
2V(\Si), \quad \textrm{ if $\Si$ is non-orientable}
\end{array} \Big\},\right. 
\end{equation}
where $V(\Si)$ denotes the volume (sometime called area) of $\Si$. Our main result is as follows.
\begin{theorem}\label{main theorem1}
Let $(M^{n+1}, g)$ be any $(n+1)$ dimensional connected closed orientable Riemannian manifold with positive Ricci curvature and $2\leq n\leq 6$. Then the min-max minimal hypersurface $\Si$ corresponding to the fundamental class $[M]$ is:
\begin{itemize}
\vspace{-5pt}
\addtolength{\itemsep}{-0.7em}
\item[$(i)$] \underline{\emph{either}} orientable of multiplicity one, which has Morse index one and $V(\Si)=W_{M}$;
\item[$(ii)$] \underline{\emph{or}} non-orientable of multiplicity two with $2V(\Si)=W_{M}$.
\end{itemize}
\vspace{-5pt}
\end{theorem}
\begin{remark}
In case $(ii)$, $\Si$ has the least area among all $\mS$. The illustrative examples for the first case are the equators $S^{n}$ embedded in $S^{n+1}$, and for the second case are the $\R\mathbb{P}^{n}$'s embedded in $\R\mathbb{P}^{n+1}$ when $n$ is an even number. Our theorem says that those are the only possible pictures.
\end{remark}

If there is no non-orientable embedded minimal hypersurface in $M$, we have the following interesting corollary.
\begin{theorem}\label{main corollary1}
Given $(M^{n+1}, g)$ as above, if $(M, g)$ has no non-orientable embedded minimal hypersurface, there is an orientable embedded minimal hypersurface $\Si^{n}\subset M^{n+1}$ with Morse index one.
\end{theorem}
\begin{remark}
If $M$ is simply connected, i.e. $\pi_{1}(M)=0$, then by  \cite[Chap 4, Theorem 4.7]{H}, there is no non-orientable embedded hypersurface in $M$. If $\pi_{1}(M)$ is finite, and the cardinality $\#\big(\pi_{1}(M)\big)$ is an odd number, then $M$ has no non-orientable embedded minimal hypersurface by looking at the universal cover.
\end{remark}

As a by-product of the proof, we have the second interesting corollary.
\begin{theorem}\label{main corollary2}
In the case of Theorem \ref{main corollary1}, the hypersurface $\Si^{n}\subset M^{n+1}$ has least area among all closed embedded minimal hypersufaces in $M^{n+1}$.
\end{theorem}
\begin{remark}
In general, compactness of  stable minimal hypersurfaces follows from curvature estimates \cite{SSY}\cite{SS}, which would imply the existence of a least area guy among the class of stable minimal hypersurfaces or even minimal hypersurfaces with uniform Morse index bound. However, the class of all closed embedded minimal hypersurfaces in $M$ does not have a priori Morse index bound. In fact, the existence of the least area minimal hypersurface comes from the min-max theory and the special structure of orientable minimal hypersurface in manifold $(M^{n+1}, g)$ with $Ric_{g}>0$.
\end{remark}

\vspace{5pt}
The main idea is as follows. The difficulty to get those geometric information is due to the fact that the min-max hypersurface is a very weak limit (varifold limit) in the construction. To overcome this difficulty, we try to get an optimal minimal hypersurface, which lies in a ``mountain pass" (see \cite{St}) type sweepout (continuous family of hypersurfaces, see Definition \ref{definition of sweepout}) in this min-max construction. Given $(M^{n+1}, g)$ as in Theorem \ref{main theorem1}, we will first embed any closed embedded minimal hypersurface $\Si$ into a good sweepout. Then such families are discretized to be adapted to the Almgren-Pitts theory. We will show that all those families lie in the same homotopy class corresponding to the fundamental class of $M$. The Almgren-Pitts theory applies to this homotopy class to produce an optimal embedded minimal hypersurface, for which we can characterize the Morse index, volume and multiplicity. There are two reasons that we must use the Almgren-Pitts theory rather than other min-max theory in continuous setting \cite{CD}\cite{DT}. One is due to the fact that the sweepouts generated by non-orientable minimal hypersurfaces (Proposition \ref{existence of good sweepout2}) do not satisfy the requirements in the continuous setting; the other reason is that only in the Almgren-Pitts setting could we show that all sweepouts lie in the same homotopy class. 

The paper is organized as follows. In Section \ref{min-max theory 1}, we give a min-max theory for manifold with boundary using the continuous setting as in \cite{DT}. In Section \ref{generating sweepout}, we show that good sweepouts can be generalized from embedded minimal hypersurface, where orientable and non-orientable cases are discussed separately. In Section \ref{min-max theory 2}, we introduce the celebrated Almgren-Pitts theory \cite{A2}\cite{P}\cite{SS}, especially the case of one parameter sweepouts. In Section \ref{discretization}, sweepouts which are continuous in the flat topology are made to discretized families in the mass norm topology as those used in the Almgren-Pitts theory. In Section \ref{orientation and multiplicity}, we give a characterization of the orientation and multiplicity of the min-max hypersurface. Finally, we prove the main result in Section \ref{main result}.

\vspace{5pt}
{\renewcommand\baselinestretch{1.0}\selectfont
\noindent\textbf{Acknowledgement:} {\small I would like to express my gratitude to my advisor Richard Schoen for lots of enlightening discussions and constant encouragement. I would also like to thank Andre Neves for pointing out an error in the first version and comments. Thanks to Brian white and Alessandro Carlotto for discussions and comments. Finally thanks to Simon Brendle, Tobias Colding, Fernando Marques, Rafe Mazzeo, William Minicozzi, Leon Simon and Gang Tian for their interests on this work.
}
\par}


\section{Min-max theory \uppercase\expandafter{\romannumeral1}---continuous setting}\label{min-max theory 1}

Let us first introduce a continuous setting for the min-max theory for constructing minimal hypersurfaces. In fact, Almgren and Pitts \cite{A2}\cite{P} used a discretized setting. They can deal with very generally discretized multi-parameter family of surfaces, but due to the discretized setting, the multi-parameter family is hard to apply to geometry directly. Later on, Smith \cite{Sm} introduced a setting using continuous families in $S^{3}$. Recently, Colding, De Lellis \cite{CD} ($n=2$) and De Lellis, Tasnady \cite{DT} ($n\geq 2$) gave a version of min-max theory using continuous category based on the ideas in \cite{Sm}. They mainly dealt with the family of level surfaces of a Morse function. Their setting is more suitable for geometric manipulation. Marques and Neves \cite{MN} extended \cite{CD} to a setting for manifolds with fixed convex boundary when $n=2$. They used that to construct smooth sweepout by Heegaard surface in certain three manifolds. In this section we will mainly use the version by De Lellis and Tasnady \cite{DT}. We will extend Marques and Neves' min-max construction for manifolds with fixed convex boundary to high dimensions. 

Let $(M^{n+1}, g)$ be a Riemannian manifold with or without boundary $\partial M$. $\mH^{n}$ denotes the $n$ dimensional Hausdorff measure. When $\Si^{n}$ is a $n$-dimensional submanifold, we use $V(\Si)$ to denote $\mH^{n}(\Si)$.
\begin{definition}\label{definition of sweepout}
A family of $\mH^{n}$ measurable closed subsets $\{\Ga_{t}\}_{t\in[0, 1]^{k}}$\footnote{The parameter space $[0, 1]$ can be any other interval $[a, b]$ in $\R$.} of $M$ with finite $\mH^{n}$ measure is called \emph{a generalized smooth family of hypersurfaces} if
\begin{itemize}
\vspace{-5pt}
\setlength{\itemindent}{1em}
\addtolength{\itemsep}{-0.7em}
\item[(s1)] For each $t$, there is a finite subset $P_{t}\subset M$, such that $\Ga_{t}$ is a smooth hypersurface in $M\setminus P_{t}$;
\item[(s2)] $t\rightarrow \mH^{n}(\Ga_{t})$ is continuous, and $t\rightarrow \Ga_{t}$ is continuous in the Hausdorff topology;
\item[(s3)] $\Ga_{t}\rightarrow \Ga_{t_{0}}$ smoothly in any compact $U\subset\subset M\setminus P_{t_{0}}$ as $t\rightarrow t_{0}$;
\vspace{-5pt}
\end{itemize}
When $\partial M=\emptyset$, a generalized smooth family $\{\Si_{t}\}_{t\in[0, 1]}$ is called a \emph{sweepout} of $M$ if there exists a family of open sets $\{\Om_{t}\}_{t\in[0, 1]}$, such that
\begin{itemize}
\vspace{-5pt}
\setlength{\itemindent}{1.5em}
\addtolength{\itemsep}{-0.7em}
\item[(sw1)] $(\Si_{t}\setminus \partial\Om_{t})\subset P_{t}$, for any $t\in[0, 1]$;
\item[(sw2)] $\textrm{Volume}(\Om_{t}\setminus\Om_{s})+\textrm{Volume}(\Om_{s}\setminus\Om_{t})\rightarrow 0$, as $s\rightarrow t$;
\item[(sw3)] $\Om_{0}=\emptyset$, and $\Om_{1}=M$.
\vspace{-5pt}
\end{itemize}
When $\partial M\neq\emptyset$, a sweepout is required to satisfy all the above except with $(sw3)$ changed by
\begin{itemize}
\vspace{-5pt}
\setlength{\itemindent}{1.5em}
\addtolength{\itemsep}{-0.7em}
\item[(sw3')] $\Om_{1}=M$. $\Si_{0}=\partial M$, $\Si_{t}\subset int(M)$ for $t>0$, and $\{\Si_{t}\}_{0\leq t\leq \ep}$ is a smooth foliation of a neighborhood of $\partial M$ for some small $\ep>0$, i.e. there exists a smooth function $w:[0, \ep]\times\partial M\rightarrow \R$, with $w(0, x)=0$ and $\frac{\partial}{\partial t}w(0, x)>0$, such that
$$\Si_{t}=\{exp_{x}\big(w(t, x)\nu(x)\big):\ x\in\partial M\},\quad \textrm{ for }t\in[0, \ep],$$
where $\nu$ is the inward unit normal for $(M, \partial M)$.
\vspace{-5pt}
\end{itemize}
\end{definition}
\begin{remark}
The first part of the definition follows from \cite[Definition 0.2]{DT}, while the second part borrows idea from \cite{MN}.
\end{remark}

We will need the following two basic results.
\begin{proposition}\label{level sets of morse function}
\emph{(\cite[Proposition 0.4]{DT})} Assume $\partial M=\emptyset$. Let $f: M\rightarrow[0, 1]$ be a smooth Morse function. Then the level sets $\big\{\{f=t\}\big\}_{t\in[0, 1]}$ is a sweepout.
\end{proposition}

Given a generalized family $\{\Ga_{t}\}$, we set
$$\bL(\{\Ga_{t}\})=\max_{t}\mH^{n}(\Ga_{t}).$$
As a consequence of the isoperimetric inequality, we have,
\begin{proposition}\label{lower bound of level sets}
\emph{(\cite[Proposition 1.4]{CD} and \cite[Proposition 0.5]{DT})} Assume $\partial M=\emptyset$. There exists a positive constant $C(M)>0$ depending only on $M$, such that $\bL(\{\Si_{t}\})\geq C(M)$ for any sweepout $\{\Si_{t}\}_{t\in[0, 1]}$.
\end{proposition}

We need the following notion of homotopy equivalence.
\begin{definition}
When $\partial M=\emptyset$, two sweepouts $\{\Si^{1}_{t}\}_{t\in[0, 1]}$ and $\{\Si^{2}_{t}\}_{t\in[0 ,1]}$ are \emph{homotopic} if there is a generalized smooth family $\{\Ga_{(s, t)}\}_{(s, t)\in[0, 1]^{2}}$, such that $\Ga_{(0, t)}=\Si^{1}_{t}$ and $\Ga_{(1, t)}=\Si^{2}_{t}$. When $\partial M\neq\emptyset$, we further require the following condition:
\begin{itemize}
\setlength{\itemindent}{0em}
\addtolength{\itemsep}{-0.7em}
\item[$(*)$] $\Ga_{(s, 0)}\equiv\partial M$, $\Ga_{(s, t)}\subset int(M)$ for $t>0$, and for some small $\ep>0$, there exits a smooth function $w:[0, \ep]\times[0, \ep]\times\partial M\rightarrow \R$, with $w(s, 0, x)=0$ and $\frac{\partial}{\partial t}w(s, 0, x)>0$, such that
$$\Ga_{(s, t)}=\{exp_{x}\big(w(s, t, x)\nu(x)\big):\ x\in\partial M\},\quad \textrm{ for } (s, t)\in[0, \ep]\times[0, \ep].$$
\end{itemize}
A family $\La$ of sweepouts is called \emph{homotopically closed} if it contains the homotopy class of each of its elements.
\end{definition}
\begin{remark}\label{remark about sweepout}
Denote $\textrm{Diff}_{0}(M)$ to be the isotopy group of diffeomorphisms of $M$. When $\partial M\neq\emptyset$, we require the isotopies to leave a neighborhood of $\partial M$ fixed. Given a sweepout $\{\Si_{t}\}_{t\in[0, 1]}$, and $\psi\in C^{\infty}([0, 1]\times M, M)$ with $\psi(t)\in \textrm{Diff}_{0}(M)$ for all $t$, then $\{\psi(t, \Si_{t})\}_{t\in[0, 1]}$ is also a sweepout, which is homotopic to $\{\Si_{t}\}$. Such homotopies will be called homotopies induced by ambient isotopies. 
\end{remark}

Given a homotopically closed family $\La$ of sweepouts, the \emph{width of $M$ associated with $\La$} is defined as,
\begin{equation}
W(M, \partial M, \La)=\inf_{\{\Si_{t}\}\in \La}\bL\big(\{\Si_{t}\}\big).
\end{equation}
When $\partial M=\emptyset$, we omit $\partial M$ and write the width as $W(M, \La)$. In case $\partial M=\emptyset$, as a corollary of Proposition \ref{lower bound of level sets}, the width of $M$ is always nontrivial, i.e. $W(M, \La)\geq C(M)>0$.

A sequence $\big\{\{\Si^{n}_{t}\}_{t\in[0, 1]}\big\}_{n=1}^{\infty}\subset \La$ of sweepouts is called a \emph{minimizing sequence} if $\F\big(\{\Si^{n}_{t}\}\big)\searrow W(M, \partial M, \La)$. A sequence of slices $\{\Si^{n}_{t_{n}}\}$ with $t_{n}\in[0, 1]$ is called a \emph{min-max sequence} if $\mH^{n}(\Si^{n}_{t_{n}})\rightarrow W(M, \partial M, \La)$. The motivation in the min-max theory \cite{P}\cite{SS}\cite{CD}\cite{DT} is to find a regular minimal hypersurface as a min-max limit corresponding to the width $W(M, \partial M, \La)$.

If $\partial M\neq \emptyset$ and $\nu$ is the inward unit normal for $(M, \partial M)$, we denote the mean curvature of the boundary by $H(\partial M)$, and the mean curvature vector by $H(\partial M)\nu$. Here the sign convention for $H$ is that $H(\partial M)(p)=-\sum_{i=1}^{n}\lan\nabla_{e_{i}}\nu, e_{i}\ran$, where $\{e_{1}, \cdots, e_{n}\}$ is an local orthonormal basis at $p\in\partial M$. Based on the main results in \cite{DT} and an idea in \cite{MN}, we have the following main result for this section.
\begin{theorem}\label{min-max existence}
Let $(M^{n+1}, g)$ be a connected compact Riemannian manifold with or without boundary $\partial M$ and $2\leq n\leq 6$. When $\partial M\neq \emptyset$, we assume $H(\partial M)>0$. For any homologically closed family $\La$ of sweepouts, with $W(M, \partial M, \La)>V(\partial M)$ if $\partial M\neq \emptyset$, there exists a min-max sequence $\{\Si^{n}_{t_{n}}\}$ of $\La$ that converges in the varifold sense to an embedded minimal hypersurface $\Si$ (possibly disconnected), which lies in the interior of $M$ if $\partial M\neq \emptyset$. Furthermore, the width $W(M, \partial M, \La)$ is equal to the volume of $\Si$ if counted with multiplicities.
\end{theorem}
\begin{proof}
When $\partial M=\emptyset$, it is just Theorem 0.7 in \cite{DT}.

Now let us assume $\partial M\neq \emptyset$. The result follows from an observation of Theorem 2.1 in \cite{MN} and minor modifications of the arguments in \cite{DT}. Here we will state the main steps and point out the key points on how to modify arguments in \cite{DT} to our setting.

\noindent\underline{\textbf{Part 1}}: Since $H(\partial M)>0$, by almost the same argument as in the proof of \cite[Theorem 2.1]{MN}, we can find $a>0$, and a minimizing sequence of sweepouts $\big\{\{\Si^{n}_{t}\}_{t\in[0, 1]}\big\}$, such that
\begin{equation}\label{inward deformation}
\mH^{n}(\Si^{n}_{t})\geq W(M, \partial M, \La)-\de,\ \Longrightarrow\ d(\Si^{n}_{t}, \partial M)\geq a/2,
\end{equation}
where $\de=\frac{1}{4}\big(W(M, \partial M, \La)-V(\partial M)\big)>0$, and $d(\cdot, \cdot)$ is the distance function of $(M, g)$.

Let us discuss the minor difference between our situation and those in \cite{MN}. Write a neighborhood of $\partial M$ using normal coordinates $[0, 2a]\times \partial M$ for some $a>0$, such that the metric $g=dr^{2}+g_{r}$. In \cite{MN} they deform an arbitrary minimizing sequence to satisfy (\ref{inward deformation}) by ambient isotopies induced by an vector field $\varphi(r)\frac{\partial}{\partial r}$, where $\varphi(r)$ is a cutoff function supported in $[0, 2a]$. Although the argument in \cite[Theorem 2.1]{MN} was given in dimension 2, it works in all dimension. The only difference is that in the proof of the claim on \cite[page 5]{MN}, we need to take the orthonormal basis $\{e_{1}, \cdots, e_{n}\}$, such that $\{e_{1}, \cdots, e_{n-1}\}$ is orthogonal to $\frac{\partial}{\partial r}$, and then projects $e_{n}$ to the orthogonal compliment of $\frac{\partial}{\partial r}$. Then all the argument follows exactly the same as in \cite{MN}.

\noindent\underline{\textbf{Part 2}}: Now let us sketch the main steps for modifying arguments of the min-max construction in \cite{DT}\cite{CD} to our setting.

Given the minimizing sequence $\big\{\{\Si^{n}_{t}\}_{t\in[0, 1]}\big\}\subset\La$ as above. The first step is a tightening process as in \cite[\S 4]{CD}, where we deform each $\{\Si^{n}_{t}\}_{t\in[0, 1]}$ to another one $\{\ti{\Si}^{n}_{t}\}_{t\in[0, 1]}$ by ambient isotopy $\{F_{t}\}_{t\in[0, 1]}\subset\textrm{Diff}_{0}(M)$, i.e. $\{\ti{\Si}^{n}_{t}=F(t, \Si^{n}_{t})\}_{t\in[0, 1]}\subset\La$, such that every min-max sequence $\{\ti{\Si}^{n}_{t_{n}}\}$ converges to a stationary varifold. Since those $\Si^{n}_{t}$ with volume near $W(M, \partial M, \La)$ have a distance $a/2>0$ away from $\partial M$, we can take all the deformation vector field to be zero near $\partial M$ in \cite[\S 4]{CD}. Hence $\{\ti{\Si}^{n}_{t}\}$ can be choose to satisfy (\ref{inward deformation}) too.

The second step is to find an almost minimizing min-max sequence (see Definition 2.3 and Proposition 2.4 in \cite{DT}) $\{\ti{\Si}^{n}_{t_{n}}\}$ among $\{\ti{\Si}^{n}_{t}\}_{t\in[0, 1]}$, where $\ti{\Si}^{n}_{t_{n}}$ converge to a stationary varifold $V$. By (\ref{inward deformation}), the slices $\ti{\Si}^{n}_{t}$ with volume near $W(M, \partial M, \La)$ always have a distance $a/2>0$ away from $\partial M$, hence they are almost minimizing in any open set supported near $\partial M$. Away from $\partial M$, all the arguments in \cite[\S 3]{DT} work, hence it implies the existence of an almost minimizing sequence in the sense of \cite[Proposition 2.4]{DT}, which are supported away from $\partial M$.

The final step is to prove that the limiting stationary varifold $V$ of the almost minimizing sequence is supported on a smooth embedded minimal hypersurface. This step was done in \cite[\S 4 and \S 5]{DT}. The arguments are purely local. By our construction, the corresponding varifold measure $|V|$ on $M$ is supported away from $\partial M$, hence the regularity results in \cite{DT} are true in our case. Counting the dimension restriction $2\leq n\leq 6$, it implies the conclusion.
\end{proof}


\section{Min-max family from embedded minimal hypersurface}\label{generating sweepout}

In this section, by exploring some special structure for embedded minimal hypersurfaces in positive Ricci curvature manifold, we will show that every embedded close connected orientable minimal hypersurface can be embedded into a sweepout, and a double cover of every embedded closed connected non-orientable minimal hypersurface can be embedded into a sweepout in a double cover of the manifold. The sweepouts constructed in both cases can be chosen to be a level surfaces of a Morse function, which hence represent the fundamental class of the ambient manifold (see Theorem \ref{identification of (1, M) homotopy class}). We first collect some results on differentiable topology. 

\begin{theorem}\label{boundary orientable}
\emph{(\cite[Chap 4, Lemma 4.1 and Theorem 4.2]{H})} Let $\Om$ be a connected compact orientable manifold with boundary $\partial\Om$. Then $\partial\Om$ is orientable.
\end{theorem}
\begin{theorem}\label{separate orientable}
\emph{(\cite[Chap 4, Theorem 4.5]{H})} Let $M$ be a connected closed orientable manifold, and $\Si\subset M$ a connected closed embedded submanifold of codimension $1$. If $\Si$ separates $M$, i.e. $M\setminus\Si$ has two connected components, then $\Si$ is orientable.
\end{theorem}
\begin{lemma}\label{normal bundle trivial}
Given $M$ and $\Si$ as above, then $\Si$ is orientable if and only if the normal bundle of $\Si$ inside $M$ is trivial.
\end{lemma}
\begin{proof}
The tangent bundle has a splitting $TM\big{|}_{\Si}=T\Si\oplus N$, where $N$ is the normal bundle. Hence our result is a corollary of Lemma 4.1 and Theorem 4.3 in \cite[Chap 4]{H}.
\end{proof}

We also need the following result which says that any two connected minimal surfaces must intersect in positive Ricci curvature manifolds.
\begin{theorem}\label{intersection}
\emph{(Generalized Hadamard Theorem in \cite{F})} Let $(M, g)$ be a connected manifold with $Ric_{g}>0$, then any two connected closed immersed minimal hypersurfaces $\Si$ and $\Si^{\pr}$ must intersect.
\end{theorem}


Let $\Si^{n}\subset M^{n+1}$ be a minimal hypersurface. When $\Si$ is two-sided, i.e. the normal bundle of $\Si$ is trivial, there always exists a unit normal vector field $\nu$.
The \emph{Jacobi operater} is 
$$L\phi=\lap_{\Si}\phi+(Ric(\nu, \nu)+|A|^{2})\phi,$$
where $\phi\in C^{\infty}(\Si)$, $\lap_{\Si}$ is the Laplacian operator on $\Si$ with respect to the induced metric, and $A$ is the second fundamental form of $\Si$. $\la\in\R$ is an \emph{eigenvalue} of $L$ if there exists a $\phi\in C^{\infty}(\Si)$, such that $L\phi=-\la\phi$. The \emph{Morse index} (abbreviated as \emph{index} in the following) of $\Si$, denoted by $ind(\Si)$, is the number of negative eigenvalues of $L$ counted with multiplicity. $\Si$ is called \emph{stable} if $ind(\Si)\geq 0$, or in another word $L$ is a nonpositive operator. Clearly $Ric_{g}>0$ implies that there is no closed two-sided stable minimal hypersurface. 

Using basic algebraic topology and geometric measure theory, together with the fact that there is no two-sided stable minimal hypersurface when $Ric_{g}>0$, we can show the reverse of Theorem \ref{separate orientable} when $2\leq n\leq 6$. 
\begin{proposition}\label{separable}
Let $(M^{n+1}, g)$ be a connected closed orientable Riemannian manifold with $2\leq n\leq 6$ and $Ric_{g}>0$, then every embedded connected closed orientable hypersurface $\Si^{n}\subset M^{n+1}$ must separate $M$.
\end{proposition}
\begin{proof}
Since $\Si^{n}$ is orientable, the fundamental class $[\Si^{n}]$ (see \cite[p.355]{B}) of $\Si$ represents a homology class in $H_{n}(M, \mZ)$. Using the language of geometric measure theory, $\Si$ is an integral $n$ cycle, hence represents an integral $n$ homology class $[\Si^{n}]$ in $H_{n}(M, \mZ)$ (see \cite[Chap4, \S 4.4]{FH}). Suppose that $\Si^{n}$ does not separate. Take a coordinates chart $U\subset M$, such that $U\cap \Si\neq\emptyset$. Since $\Si^{n}$ is embedded, $\Si$ separates $U$ into $U_{1}$ and $U_{2}$ after possibly shrinking $U$. Pick $p_{1}\in U_{1}$ and $p_{2}\in U_{2}$. We can connect $p_{1}$ to $p_{2}$ by a curve $\ga_{1}$ inside $U$, such that $\ga_{1}$ intersects $\Si$ transversally only once. Since $\Si$ does not separate, $M\setminus \Si$ is connected. We can connect $p_{1}$ to $p_{2}$ by a curve $\ga_{2}$ inside $M\setminus\Si$. Now we get a closed curve $\ga=\ga_{1}\cup\ga_{2}$, which intersects $\Si$ transversally only once. Hence $\Si$ meets $\ga$ transversally, and $\Si\cap\ga$ is a single point. Using the intersection theory (see \cite[page 367]{B}), the intersection of the $n$ homology $[\Si]$ and the $1$ homology $[\ga]$ is
$$[\Si]\cdot[\ga]=[\Si\cap\ga]\neq 0.$$
Hence $[\Si]\neq 0$ in $H_{n}(M, \mZ)$. Now we can minimize the mass inside the integral homology class $[\Si]$ (as a collection of integral cycles). \cite[Lemma 34.3]{S} tells us that there is a minimizing integral current $T_{0}\in[\Si]$. Moreover, the codimention one regularity theory (\cite[Theorem 37.7]{S}) when $2\leq n\leq 6$ implies that $T_{0}$ is represented by a smooth $n$ dimensional hypersurface $\Si_{0}$ (possibly with multiplicity) , i.e. $T_{0}=m[\Si_{0}]$, where $m\in\mZ$, $m\neq 0$. Since $m[\Si_{0}]$ represents a nontrivial integral homology class, $\Si_{0}$ is orientable. The fact that both $M$ and $\Si_{0}$ are orientable implies that the normal bundle of $\Si_{0}$ is trivial by Lemma \ref{normal bundle trivial}, hence $\Si_{0}$ is two-sided. By the nature of mass minimizing property of $T$, $\Si_{0}$ must be locally volume minimizing, hence $\Si_{0}$ is stable. This is a contradiction with $Ric_{g}>0$.
\end{proof}

From now on, we always assume that $(M^{n+1}, g)$ is connected closed oriented with $2\leq n\leq 6$, and hypersurfaces $\Si^{n}\subset M^{n+1}$ are connected closed and embedded in this section.



\subsection{Orientable case}\label{orientable case}

The following proposition is a key observation in proving our main theorem, which asserts that every orientable minimal hypersurface lies in a good sweepout in manifold $(M^{n+1}, g)$ of positive Ricci curvature when $2\leq n\leq 6$. Denote
\begin{equation}\label{S+}
\mS_{+}=\{\Si^{n}\subset(M^{n+1}, g):\ \Si^{n} \textrm{ is an orientable minimal hypersurface in } M\}.
\end{equation}
\begin{proposition}\label{existence of good sweepout}
For any $\Si\in\mS_{+}$, there exists a sweepout $\{\Si_{t}\}_{t\in[-1, 1]}$ of $M$, such that
\begin{itemize}
\vspace{-5pt}
\setlength{\itemindent}{1em}
\addtolength{\itemsep}{-0.7em}
\item[$(a)$] $\Si_{0}=\Si$;
\item[$(b)$] $\mH^{n}(\Si_{t})\leq V(\Si)$, with equality only if $t=0$;
\item[$(c)$] $\{\Si_{t}\}_{t\in[-\ep, \ep]}$ forms a smooth foliation of a neighborhood of $\Si$, i.e. there exists $w(t, x)\in C^{\infty}([-\ep, \ep]\times \Si)$, $w(0, x)=0$, $\frac{\partial}{\partial t}w(0, x)>0$, such that
$$\Si_{t}=\{exp_{x}\big(w(t, x)\nu(x)\big):\ x\in\Si\},\quad t\in [-\ep, \ep],$$
where $\nu$ is the unit normal vector field of $\Si$ in $M$.
\vspace{-5pt}
\end{itemize}
\end{proposition}
\begin{proof}
By Proposition \ref{separable}, $\Si$ separates $M$, hence $M\setminus \Si=M_{1}\cup M_{2}$ is a disjoint union of two connected components $M_{1}$ and $M_{2}$, with $\partial M_{1}=\partial M_{2}=\Si$. Assume that the unit normal vector field $\nu$ points into $M_{1}$. We denote $\la_{1}$ to be the first eigenvalue of the Jacobi operator $L$, and $u_{1}$ the corresponding eigenfunction. The first eigenvalue has multiplicity 1, and $u_{1}>0$ everywhere on $\Si$. $Ric_{g}>0$ means that $\Si$ is unstable, hence $\la_{1}<0$, i.e. $Lu_{1}=-\la_{1}u_{1}>0$.

Consider the local foliation by the first eigenfunction via exponential map,
$$\Si_{s}=\{exp_{x}\big(su_{1}(x)\nu(x)\big):\ x\in\Si\},\quad s\in[-\ep, \ep].$$
\begin{itemize}
\setlength{\itemindent}{1em}
\addtolength{\itemsep}{-0.7em}
\item For $\ep>0$ small enough, since $u_{1}>0$, the map $F: [-\ep, \ep]\times\Si\rightarrow M$ given by $F(s, x)=exp_{x}\big(su_{1}(x)\nu(x)\big)$ is a smooth diffeomorphic one to one map, hence $\{\Si_{s}\}_{s\in[-\ep, \ep]}$ is a smooth foliation of a neighborhood of $\Si$.
\item Since $u_{1}>0$, $\Si_{s}$ is contained in $M_{1}$ (in $M_{2}$) for $0<s<\ep$ (for $-\ep<s<0$).
\item By the first and second variational formulae (see \cite{CM2}\cite{S}),
$$\frac{d}{ds}\bigg{|}_{s=0}V(\Si_{s})=-\int_{\Si}Hu_{1}d\mu=0,\quad \frac{d^{2}}{ds^{2}}\bigg{|}_{s=0}V(\Si_{s})=-\int_{\Si}u_{1}Lu_{1}d\mu<0,$$
where $H\equiv 0$ is the mean curvature of $\Si$. So $V(\Si_{s})\leq V(\Si)$ for $s\in[-\ep, \ep]$, and equality holds only if $s=0$.
\item Denote $\vec{H}_{s}$ to be the mean curvature operator of $\Si_{s}$, then
$$\frac{\partial}{\partial s}\bigg{|}_{s=0}\lan\vec{H}_{s}, \nu\ran=Lu_{1}>0.$$
Hence $H(\Si_{s})>0$ for $0<s<\ep$ with respect to the normal $\nu$ for $\ep$ small enough.
\end{itemize}

Denote $M_{1, s_{0}}=M_{1}\setminus\{\Si_{s}\}_{0\leq s\leq s_{0}}$ for $0<s_{0}\leq \ep$, which is the region bounded by $\Si_{s_{0}}$. Similarly we have $M_{2, s_{0}}$, such that $\partial M_{2, s_{0}}=\Si_{-s_{0}}$. We will extend the foliation $\{\Si_{s}\}$ to $M_{1, s_{0}}$ and $M_{2, s_{0}}$. We need the following claim, which is proved in Appendix \ref{appendix},
\begin{claim}\label{claim for existence of good sweepout}
For $\ep$ small enough, there exists a sweepout $\{\ti{\Si}_{s}\}_{s\in[-1, 1]}$, such that $\ti{\Si}_{s}=\Si_{s}$ for $s\in [-\frac{1}{2}\ep, \frac{1}{2}\ep]$, and $\ti{\Si}_{s}\subset M_{1, \frac{1}{2}\ep}$ (or $\subset M_{2, \frac{1}{2}\ep}$) when $s>\frac{1}{2}\ep$ (or $s<-\frac{1}{2}\ep$).
\end{claim}
Now cut out part of the sweepout $\{\ti{\Si}_{s}\}_{s\in[\frac{1}{4}\ep, 1]}$, which is then a sweepout of $(M_{1, \frac{1}{4}\ep}, \partial M_{1, \frac{1}{4}\ep})$ (abbreviated as $(M_{1}, \partial M_{1})$) by Definition \ref{definition of sweepout}. Consider the smallest homotopically closed family $\ti{\La}_{1}$ of sweepouts containing $\{\ti{\Si}_{s}\}_{s\in[\frac{1}{4}\ep, 1]}$. If the width $W(M_{1}, \partial M_{1}, \ti{\La}_{1})>V(\partial M_{1})$, then by Theorem \ref{min-max existence} and the fact that $H(\partial M_{1})=H(\Si_{\frac{1}{4}\ep})>0$, there is a nontrivial embedded minimal hypersurface $\ti{\Si}$ lying in the interior of $M_{1}$, which is then disjoint with $\Si$, hence is a contradiction to Theorem \ref{intersection}. So $W(M_{1}, \partial M_{1}, \ti{\La}_{1})\leq V(\partial M_{1})$, which means that there exists sweepouts $\{\ti{\Si}^{\pr}_{s}\}_{s\in[\frac{1}{4}\ep, 1]}$ of $(M_{1}, \partial M_{1})$, with $\max_{s\in[\frac{1}{4}\ep, 1]}\mH^{n}(\ti{\Si}^{\pr}_{s})$ very close to $V(\partial M_{1})$. Since $\partial M_{1}=\Si_{\frac{1}{4}\ep}$, and $V(\Si_{\frac{1}{4}\ep})<V(\Si)$ by our construction above, we can pick up one sweepout $\{\ti{\Si}^{\pr}_{s}\}_{s\in[\frac{1}{4}\ep, 1]}$ with $\max_{s\in[\frac{1}{4}\ep, 1]}\mH^{n}(\ti{\Si}^{\pr}_{s})<V(\Si)$. 

We can do similar things to $M_{2, \frac{1}{4}\ep}$ to get another partial sweepout. Then we finish by putting them together with $\{\Si_{s}\}_{s\in[-\frac{1}{4}\ep, \frac{1}{4}\ep]}$.
\end{proof}

\subsection{Non-orientable case}\label{non-orientable}

We have the following topological characterization of non-orientable embedded hypersurfaces in orientable manifold $M$. 
\begin{proposition}\label{double cover orientation}
For any non-orientable embedded hypersurface $\Si^{n}$ in an orientable manifold $M^{n+1}$, there exists a connected double cover $\ti{M}$ of $M$, such that the lifts $\ti{\Si}$ of $\Si$ is a connected orientable embedded hypersurface. Furthermore, $\ti{\Si}$ separates $\ti{M}$, and both components of $\ti{M}\setminus\ti{\Si}$ are diffeomorphic to $M\setminus\Si$.
\end{proposition}
\begin{proof}
Since $\Si$ is non-orientable, hence $M\setminus\Si$ is connected by Theorem \ref{separate orientable}. Denote $\Om=M\setminus\Si$. $\Om$ has a topological boundary $\partial\Om$. $\Om$ is orientable since $M$ is orientable, hence $\partial\Om$ is orientable by Theorem \ref{boundary orientable}.
\begin{claim}
$\partial\Om$ is a double cover of $\Si$.
\end{claim}
This is proved as follows. $\forall x\in\Si$, there exists a neighborhood $U$ of $x$, i.e. $x\in U\subset M$, with $U$ diffeomorhic to a unit ball $B_{1}(0)$. Since $\Si$ is embedded, after possibly shrinking $U$, $\Si\cap U$ is a topological $n$ dimensional ball, and $\Si$ separates $U$ into two connected components $U_{1}$ and $U_{2}$, i.e. $U\setminus\Si=U_{1}\cup U_{2}$. Then the sets $U\cap\Si\simeq (\partial U_{1})\cap\Si\simeq (\partial U_{2})\cap\Si$ are diffeomorphic. The sets $\{U\cap\Si\}$ form a system of local coordinate charts for $\Si$. Moreover $\{(\partial U_{1})\cap\Si, (\partial U_{2})\cap\Si\}$ form a systems of local coordinate charts for $\partial\Om$, and $\{(U_{1}, \partial U_{1}\cap\Si), (U_{2}, \partial U_{2}\cap\Si)\}$ form a systems of local boundary coordinate charts for $(\Om, \partial\Om)$. Hence $\partial\Om$ is a double covering of $\Si$, with the covering map given by $(\partial U_{1})\cap\Si, (\partial U_{2})\cap\Si)\rightarrow U\cap\Si$.

Since $\Si$ is connected, $\partial\Om$ has no more than two connected components. If $\partial\Om$ is not connected, then $\partial\Om$ has two connected components, i.e. $\partial\Om=(\partial\Om)_{1}\cup(\partial\Om)_{2}$, with $\Si\simeq(\partial\Om)_{1}\simeq(\partial\Om)_{2}$. Hence $\Si$ is orientable since $\partial\Om$ is orientable, which is a contradiction. So $\Om$ must be connected. Let $\ti{M}=\Om\sqcup_{\{\partial\Om: x\rightarrow x^{*}\}}\Om$ be the gluing of two copies of $(\Om, \partial\Om)$ along $\partial\Om$ using the deck transformation map $x\rightarrow x^{*}$ of the covering $\partial\Om\rightarrow\Si$, then the lift of $\Si$ is $\ti{\Si}\simeq\partial\Om$. $\ti{M}$ is then orientable and satisfies all the requirements.
\end{proof}

As a direct corollary of the results in the previous section, we can embed a double cover of a non-orientable minimal hypersurface to a sweepout in the double cover $\ti{M}$ of a manifold $(M^{n+1}, g)$ with positive Ricci curvature when $2\leq n\leq 6$. Let
\begin{equation}\label{S-}
\mS_{-}=\{\Si^{n}\subset(M^{n+1}, g):\ \Si^{n} \textrm{ is a non-orientable minimal hypersurface in } M\}.
\end{equation}
\begin{proposition}\label{existence of good sweepout2}
Given $\Si\in\mS_{-}$, there exists a family $\{\Si_{t}\}_{t\in[0, 1]}$ of closed sets, such that
\begin{itemize}
\vspace{-5pt}
\addtolength{\itemsep}{-0.7em}
\item[$(a)$] $\Si_{0}=\emptyset$;
\item[$(b)$] $\{\Si_{t}\}_{t\in[0, 1]}$ satisfies (s1)(sw1)(sw2)(sw3) in Definition \ref{definition of sweepout};
\item[$(c)$] $\max_{t\in[0, 1]}\mH^{n}(\Si_{t})=2V(\Si)$ and $\mH^{n}(\Si_{t})<2V(\Si)$ for all $t\in[0, 1]$;
\item[$(d)$] (s2) in Definition \ref{definition of sweepout} only fails when $t\rightarrow 0$, where $\mH^{n}(\Si_{t})\rightarrow 2V(\Si)$;
\item[$(e)$] (s3) in Definition \ref{definition of sweepout} only fails when $t\rightarrow 0$, where $\Si_{t}\rightarrow 2\Si$.
\end{itemize}
\end{proposition}
\begin{proof}
Consider the double cover $(\ti{M}, g)$ given by Proposition \ref{double cover orientation}. The lift $\ti{\Si}$ is an orientable minimal hypersurface, and must has the double volume of $\Si$, i.e. $V(\ti{\Si})=2V(\Si)$. $\ti{\Si}$ separates $\ti{M}$ into two isomorphic components $\ti{M}_{1}$ and $\ti{M}_{2}$, which are both isomorphic to $M\setminus\Si$. We can apply Proposition \ref{existence of good sweepout} to $(\ti{M}, \ti{\Si})$ to get a sweepout $\{\ti{\Si}_{t}\}_{t\in[-1, 1]}$ satisfying (a)(b)(c) there. By the construction, we know that $\ti{\Si}_{t}\subset M_{1}$ for $t>0$, and $\ti{\Si}_{t}\subset M_{2}$ for $t<0$. To define $\{\Si_{t}\}_{t\in[0, 1]}$, we can let $\Si_{t}=\ti{\Si}_{t}$ while identifying $M_{1}$ with $M\setminus\Si$, and let $\Si_{0}=\emptyset$. Then the properties follow from those of $\{\ti{\Si}_{t}\}_{t\in[-1, 1]}$.
\end{proof}


\section{Min-max theory \uppercase\expandafter{\romannumeral2}---Almgren-Pitts discrete setting}\label{min-max theory 2}

Let us introduce the min-max theory developed by Almgren and Pitts \cite{A1}\cite{A2}\cite{P}. We will briefly give the notations in \cite[\S 4.1]{P} in order to state the min-max theorem. Marques and Neves also gave a nice introduction in \cite[\S 7 and \S 8]{MN2}. For notations in geometric measure theory, we refer to \cite{S}, \cite[\S 2.1]{P} and \cite[\S 4]{MN2}.

Fix an oriented Riemannian manifold $(M^{n+1}, g)$ of dimension $n+1$, with $2\leq n\leq 6$. Assume that $(M^{n+1}, g)$ is embedded in some $\R^{N}$ for $N$ large. We denote by $\bI_{k}(M)$ the space of $k$-dimensional integral currents in $\R^{N}$ with support in $M$; $\Z_{k}(M)$ the space of integral currents $T\in\bI_{k}(M)$ with $\partial T=0$; and $\V_{k}(M)$ the space $k$-dimensional rectifiable varifolds in $\R^{N}$ with support in $M$, endowed with the weak topology. Given $T\in\bI_{k}(M)$, $|T|$ and $\|T\|$ denote the integral varifold and Radon measure in $M$ associated with $T$ respectively. $\F$ and $\M$ denote the flat norm and mass norm on $\bI_{k}(M)$ respectively. $\bI_{k}(M)$ and $\Z_{k}(M)$ are in general assumed to have the flat norm topology. $\bI_{k}(M, \M)$ and $\Z_{k}(M, \M)$ are the same space endowed with the mass norm topology. Given a smooth surface $\Si$ or an open set $\Om$ as in Definition \ref{definition of sweepout}, we use $[[\Si]]$, $[[\Om]]$ and $[\Si]$, $[\Om]$ to denote the corresponding integral currents and integral varifolds respectively.

We are mainly interested in the application of the Almgren-Pitts theory to the special case $\pi_{1}\big(\Z_{n}(M^{n+1}), \{0\}\big)$, so our notions will be restricted to this case.
\begin{definition}\label{cell complex}
(cell complex of $I=[0, 1]$)
\begin{itemize}
\vspace{-5pt}
\addtolength{\itemsep}{-0.7em}
\item[$1.$] $I=[0, 1]$, $I_{0}=\{[0], [1]\}$;
\item[$2.$] For $j\in\N$, $I(1, j)$ is the cell complex of $I$, whose $1$-cells are all interval of form $[\frac{i}{3^{j}}, \frac{i+1}{3^{j}}]$, and $0$-cells are all points $[\frac{i}{3^{j}}]$. Denote $I(1, j)_{p}$ the set of all $p$-cells in $I(1, j)$, with $p=0, 1$, and $I_{0}(1, j)=\{[0], [1]\}$ the boundary $0$-cells;
\item[$3.$] Given $\al$ a $1$-cell in $I(1, j)$ and $k\in\N$, $\al(k)$ denotes the $1$-dimensional sub-complex of $I(1, j+k)$ formed by all cells contained in $\al$, and $\al(k)_{0}$ are the boundary $0$-cells of $\al$; 
\item[$4.$] The boundary homeomorphism $\partial: I(1, j)\rightarrow I(1, j)$ is given by $\partial[a, b]=[b]-[a]$ and $\partial[a]=0$; 
\item[$5.$] The distance function $d: I(1, j)_{0}\times I(1, j)_{0}\rightarrow\mZ^{+}$ is defined as $d(x, y)=3^{j}|x-y|$;
\item[$6.$] The map $n(i, j): I(1, i)_{0}\rightarrow I(1, j)_{0}$ is defined as: $n(i, j)(x)\in I(1, j)_{0}$ is the unique element, such that $d\big(x, n(i, j)(x)\big)=\inf\big\{d(x, y): y\in I(1, j)_{0}\big\}$.
\end{itemize}
\end{definition}

Consider a map to the space of integral cycles: $\phi: I(1, j)_{0}\rightarrow\Z_{n}(M^{n+1})$. The \emph{fineness of $\phi$} is defined as:
\begin{equation}\label{fineness}
\mf(\phi)=\sup\Big\{\frac{\M\big(\phi(x)-\phi(y)\big)}{d(x, y)}:\ x, y\in I(1, j)_{0}, x\neq y\Big\}.
\end{equation}
A map $\phi: I(1, j)_{0}\rightarrow\big(\Z_{n}(M^{n+1}), \{0\}\big)$ means that $\phi\big(I(1, j)_{0}\big)\subset\Z_{n}(M^{n+1})$ and $\phi|_{I_{0}(1, j)_{0}}=0$, i.e. $\phi([0])=\phi([1])=0$.

\begin{definition}\label{homotopy for maps}
Given $\de>0$ and $\phi_{i}: I(1, k_{i})_{0}\rightarrow\big(\Z_{n}(M^{n+1}), \{0\}\big)$, $i=1, 2$. We say \emph{$\phi_{1}$ is $1$-homotopic to $\phi_{2}$ in $\big(\Z_{n}(M^{n+1}), \{0\}\big)$ with fineness $\de$}, if $\exists\ k_{3}\in\N$, $k_{3}\geq\max\{k_{1}, k_{2}\}$, and
$$\psi: I(1, k_{3})_{0}\times I(1, k_{3})_{0}\rightarrow \Z_{n}(M^{n+1}),$$
such that
\begin{itemize}
\vspace{-5pt}
\setlength{\itemindent}{1em}
\addtolength{\itemsep}{-0.7em}
\item $\mf(\psi)\leq \de$;
\item $\psi([i-1], x)=\phi_{i}\big(n(k_{3}, k_{i})(x)\big)$;
\item $\psi\big(I(1, k_{3})_{0}\times I_{0}(1, k_{3})_{0}\big)=0$.
\end{itemize}
\end{definition}

\begin{definition}\label{(1, M) homotopy sequence}
A \emph{$(1, \M)$-homotopy sequence of mappings into $\big(\Z_{n}(M^{n+1}), \{0\}\big)$} is a sequence of mappings $\{\phi_{i}\}_{i\in\N}$,
$$\phi_{i}: I(1, k_{i})_{0}\rightarrow\big(\Z_{n}(M^{n+1}), \{0\}\big),$$
such that $\phi_{i}$ is $1$-homotopic to $\phi_{i+1}$ in $\big(\Z_{n}(M^{n+1}), \{0\}\big)$ with fineness $\de_{i}$, and
\begin{itemize}
\vspace{-5pt}
\setlength{\itemindent}{1em}
\addtolength{\itemsep}{-0.7em}
\item $\lim_{i\rightarrow\infty}\de_{i}=0$;
\item $\sup_{i}\big\{\M(\phi_{i}(x)):\ x\in I(1, k_{i})_{0}\big\}<+\infty$.
\end{itemize}
\end{definition}

\begin{definition}\label{homotopy for sequences}
Given $S_{1}=\{\phi^{1}_{i}\}_{i\in\N}$ and $S_{2}=\{\phi^{2}_{i}\}_{i\in\N}$ two $(1, \M)$-homotopy sequence of mappings into $\big(\Z_{n}(M^{n+1}), \{0\}\big)$. \emph{$S_{1}$ is homotopic with $S_{2}$} if $\exists\ \{\de_{i}\}_{i\in\N}$, such that
\begin{itemize}
\vspace{-5pt}
\setlength{\itemindent}{1em}
\addtolength{\itemsep}{-0.7em}
\item $\phi^{1}_{i}$ is $1$-homotopic to $\phi^{2}_{i}$ in $\big(\Z_{n}(M^{n+1}), \{0\}\big)$ with fineness $\de_{i}$;
\item $\lim_{i\rightarrow \infty}\de_{i}=0$.
\end{itemize}
\end{definition}

The relation ``is homotopic with" is an equivalent relation on the space of $(1, \M)$-homotopy sequences of mapping into $\big(\Z_{n}(M^{n+1}), \{0\}\big)$ (see \cite[\S 4.1.2]{P}). An equivalent class is a \emph{$(1, \M)$ homotopy class of mappings into $\big(\Z_{n}(M^{n+1}), \{0\}\big)$}. Denote the set of all equivalent classes by $\pi^{\#}_{1}\big(\Z_{n}(M^{n+1}, \M), \{0\}\big)$. Similarly we can define the $(1, \F)$-homotopy class, and denote the set of all equivalent classes by $\pi^{\#}_{1}\big(\Z_{n}(M^{n+1}, \F), \{0\}\big)$. In fact, Almgren-Pitts showed that they are all isomorphic to the top homology group.

\begin{theorem}\label{isomorphism}
\emph{(\cite[Theorem 13.4]{A1} and \cite[Theorem 4.6]{P})} The followings are all isomorphic:
$$H_{n+1}(M^{n+1}),\ \pi^{\#}_{1}\big(\Z_{n}(M^{n+1}, \M), \{0\}\big),\ \pi^{\#}_{1}\big(\Z_{n}(M^{n+1}, \F), \{0\}\big).$$
\end{theorem}

\begin{definition}
(Min-max definition) Given $\Pi\in\pi^{\#}_{1}\big(\Z_{n}(M^{n+1}, \M), \{0\}\big)$, define:
$$\bL: \Pi\rightarrow\R^{+}$$
as a function given by:
$$\bL(S)=\bL(\{\phi_{i}\}_{i\in\N})=\limsup_{i\rightarrow\infty}\max\big\{\M\big(\phi_{i}(x)\big):\ x \textrm{ lies in the domain of $\phi_{i}$}\big\}.$$
The \emph{width of $\Pi$} is defined as
\begin{equation}\label{width}
\bL(\Pi)=\inf\{\bL(S):\ S\in\Pi\}.
\end{equation}
$S\in\Pi$ is call a \emph{critical sequence}, if $\bL(S)=\bL(\Pi)$. Let $K: \Pi\rightarrow\{\textrm{compact subsets of $\V_{n}(M^{n+1})$}\}$ be defined by
$$K(S)=\{V:\ V=\lim_{j\rightarrow\infty}[\phi_{i_{j}}(x_{j})]:\ \textrm{$x_{j}$ lies in the domain of $\phi_{i_{j}}$}\}.$$
A \emph{critical set} of $S$ is $C(S)=K(S)\cap\{V:\ \M(V)=\bL(S)\}$.
\end{definition}

The celebrated min-max theorem of Almgren-Pitts (Theorem 4.3, 4.10, 7.12, Corollary 4.7 in \cite{P}) and Schoen-Simon (for $n=6$ \cite[Theorem 4]{SS}) is as follows.
\begin{theorem}\label{AP min-max theorem}
Given a nontrivial $\Pi\in\pi^{\#}_{1}\big(\Z_{n}(M^{n+1}, \M), \{0\}\big)$, then $\bL(\Pi)>0$, and there exists a stationary integral varifold $\Si$, whose support is a closed smooth embedded minimal hypersurface (which may be disconnected with multiplicity), such that
$$\|\Si\|(M)=\bL(\Pi).$$
In particular, $\Si$ lies in the critical set $C(S)$ of some critical sequence.
\end{theorem}


\section{Discretization}\label{discretization}

In this section, we will adapt the families constructed in Section \ref{generating sweepout} to the Almgren-Pitts setting. The families constructed in Section \ref{generating sweepout} are continuous under the flat norm topology, but Almgren-Pitts theory applies only to discrete family continuous under the mass norm topology. So we need to discretize our families and to make them continuous under the mass norm. Similar issue was considered in the celebrated proof of the Willmore conjecture \cite{MN2}. Besides that, we will show that all the discretized families belong to the same homotopy class. The proof is elementary but relatively long. A fist read can cover only the statements of Proposition \ref{generating min-max family in flat topology}, Theorem \ref{generating (1, M) homotopy sequence} and Theorem \ref{identification of (1, M) homotopy class}.

\subsection{Generating min-max family}\label{generating min-max family}

\begin{proposition}\label{continuity and no mass concentration}
Given $\Phi: [0, 1]\rightarrow\Z_{n}(M^{n+1})$ defined by
$$\Phi(x)=[[\partial\Om_{x}]],\quad x\in[0, 1],$$
where $\{\Om_{t}\}_{t\in[0, 1]}$ is a family of open sets satisfies (sw1)(sw2)(sw3) in Definition \ref{definition of sweepout} for some $\{\Si_{t}\}_{t\in[0, 1]}$ satisfying (s1)(s2)(s3) there, then
\begin{itemize}
\vspace{-5pt}
\addtolength{\itemsep}{-0.7em}
\item[$(1)$] $\Phi: [0, 1]\rightarrow\big(\Z_{n}(M^{n+1}), \{0\}\big)$ is continuous under the flat topology;
\item[$(2)$] $\m(\Phi, r)=\sup\big\{\|\Phi(x)\|B(p, r):\ p\in M, x\in[0, 1]\big\}\footnote{The concept $\m$ first appears in \cite[\S 4.2]{MN2}}\rightarrow 0$ when $r\rightarrow 0$, where $B(p, r)$ is the geodesic ball of radius $r$ and centered at $p$ on $M$.
\end{itemize}
\end{proposition}
\begin{proof}
By (sw1) and (s1) in Definition \ref{definition of sweepout}, $\partial\Om_{x}$ is smooth away from finitely many points, hence it lies in $\Z_{n}(M^{n+1})$. By (sw3), $\Om_{0}=\emptyset$, $\Om_{1}=M$ implies that $\Phi(0)=\Phi(1)=0$. So $\Phi$ is well-defined as a map to $\big(\Z_{n}(M^{n+1}), \{0\}\big)$.

From the definition of flat norm (see \cite[\S 31]{S}),
$$\F\big(\Phi(x), \Phi(y)\big)\leq\|\Om_{y}-\Om_{x}\|(M)=\textrm{Volume}(\Om_{y}\Delta\Om_{x})\rightarrow 0,$$
as $y\rightarrow x$ by (sw2) in Definition \ref{definition of sweepout}. Here and in the following, we abuse $\Om$ and $\Si$ with the associated integral currents $[[\Om]]$ and $[[\Si]]$.

So what left is the last property, i.e. $\m(\Phi, r)\rightarrow 0$ when $r\rightarrow 0$. Now we will abuse the notation and write $\Phi(x)=\Si_{x}=\partial\Om_{x}$ since they only differ by a finite set of points.
\begin{lemma}
Fix $x\in[0, 1]$, and let $P_{x}$ be the finite set of singular points of $\Si_{x}$, and $B_{r}(P_{x})$ the collection of geodesic balls centered at $P_{x}$ on $M$, then $\lim_{r\rightarrow 0}\|\Si_{x}\|\big(B_{r}(P_{x})\big)=0$\footnote{Here $\|\Si\|$ is the Radon measure corresponding to the integral current $[[\Si]]$ associated with $\Si$ (see \cite[\S 27]{S}).}.
\end{lemma}
\begin{proof}
We only need to show that $\lim_{r\rightarrow 0}\|\Si_{x}\|\big(B_{r}(p)\big)=0$ for every $p\in P_{x}$. By the definition of Hausdorff measure (see \cite[\S 2]{S}), $(\mH^{n}\lc\Si_{x})(\{p\})=\mH^{n}(\Si_{x}\cap\{p\})=\mH^{n}(\{p\})=0$. Since $\mH^{n}(\Si_{x})<+\infty$, by the basic convergence property for Radon measures (see \cite[\S 11.1, Proposition 2.1]{R}),
$$0=(\mH^{n}\lc\Si_{x})(\{p\})=\lim_{r\rightarrow 0}(\mH^{n}\lc\Si_{x})\big(B_{r}(p)\big)=\lim_{r\rightarrow 0}\|\Si_{x}\|\big(B_{r}(p)\big).$$
\end{proof}

Given $r_{0}>0$ small enough, define $f:[0, r_{0}]\times M\times [0, 1]\rightarrow\R^{+}$ by
$$f(r, p, x)=\|\Si_{x}\|\big(B_{r}(p)\big).$$
\begin{lemma}
$f$ is continuous.
\end{lemma}
\begin{proof}
For the continuity of the parameter $``x"$, we can fix the ball $B_{r}(p)$. For any $\ep>0$, we can take $0<r_{x, \ep}\ll1$, such that $\|\Si_{x}\|\big(B_{r_{x, \ep}}(P_{x})\big)<\frac{\ep}{4}$ by the previous lemma, where $P_{x}$ is the finite singular set of $\Si_{x}$. Since $\Si_{y}$ converges to $\Si_{x}$ smooth on compact sets of $M\setminus P_{x}$ by $(s3)$ of Definition \ref{definition of sweepout}, we can find $\de_{x, \ep}$, such that whenever $|y-x|<\de_{x, \ep}$,
$$\big{|}\|\Si_{y}\|\big(B_{r}(p)\setminus B_{r_{x, \ep}}(P_{x})\big)-\|\Si_{x}\|\big(B_{r}(p)\setminus B_{r_{x, \ep}}(P_{x})\big)\big{|}<\frac{\ep}{4}.$$

We claim that $\|\Si_{y}\|\big(B_{r_{x, \ep}}(P_{x})\big)<\frac{\ep}{2}$ if $\de_{x, \ep}$ is small enough. Suppose not, then for a subsequence $y_{i}\rightarrow x$, $\|\Si_{y_{i}}\|\big(B_{r_{x, \ep}}(P_{x})\big)\geq\frac{\ep}{2}$. Notice $(s2)$ in Definition \ref{definition of sweepout}, i.e. $\mH^{n}(\Si_{y})\rightarrow\mH^{n}(\Si_{x})$. Now
$$\mH^{n}(\Si_{y})=\mH^{n}\big(\Si_{y}\setminus B_{r_{x, \ep}}(P_{x})\big)+\mH^{n}\big(\Si_{y}\cap B_{r_{x, \ep}}(P_{x})\big),$$
$$\mH^{n}(\Si_{x})=\mH^{n}\big(\Si_{x}\setminus B_{r_{x, \ep}}(P_{x})\big)+\mH^{n}\big(\Si_{x}\cap B_{r_{x, \ep}}(P_{x})\big).$$
Since $\Si_{y}$ converge smoothly to $\Si_{x}$ on compact subsets of $M\setminus P_{x}$, $\mH^{n}\big(\Si_{y}\setminus B_{r_{x, \ep}}(P_{x})\big)\rightarrow \mH^{n}\big(\Si_{x}\setminus B_{r_{x, \ep}}(P_{x})\big)$, hence we get a contradiction since $\mH^{n}\big(\Si_{y}\cap B_{r_{x, \ep}}(P_{x})\big)-\mH^{n}\big(\Si_{x}\cap B_{r_{x, \ep}}(P_{x})\big)>\frac{\ep}{2}-\frac{\ep}{4}=\frac{\ep}{4}$.

Combing all above, we have $\big{|}\|\Si_{y}\|\big(B_{r}(p)\big)- \|\Si_{x}\|\big(B_{r}(p)\big)\big{|}<\ep$ whenever $|y-x|<\de_{x, \ep}$, and hence proved the continuity of $f$ w.r.t. $``x"$.

For the continuity of the parameter $``r"$, we can fix $\Si_{x}$ and the point $p\in M$. For any $\ep>0$, take $r_{x, \ep}$ as above. For any $\De r>0$,
$$\mH^{n}\big(\Si_{x}\cap B_{r+\De r}(p)\big)-\mH^{n}\big(\Si_{x}\cap B_{r}(p)\big)\leq \mH^{n}\big(\Si_{x}\cap B_{r_{x, \ep}}(P_{x})\big)+\mH^{n}\big(\Si_{x}\cap A(p, r, r+\De r)\setminus B_{r_{x, \ep}}(P_{x})\big),$$
where $A(p, r, r+\De r)$ is the closed annulus. Since $\Si_{x}$ is smooth on $M\setminus P_{x}$ by $(s1)$ in Definition \ref{definition of sweepout}, we can take $\de_{x, \ep}>0$, such that whenever $\De r<\de_{x, \ep}$, $\mH^{n}\big(\Si_{x}\cap A(p, r, r+\De r)\setminus B_{r_{x, \ep}}(P_{x})\big)<\frac{\ep}{4}$. Hence $\mH^{n}\big(\Si_{x}\cap B_{r+\De r}(p)\big)-\mH^{n}\big(\Si_{x}\cap B_{r}(p)\big)<\frac{\ep}{2}$. Similar argument holds for $\De r<0$.

The continuity of the parameter $``p"$ follows exactly the same as that of $``r"$, so we omit the details here.
\end{proof}
Let us go back to the proof of $\lim_{r\rightarrow 0}\m(\Phi, r)=0$. Since $[0, r_{0}]\times M\times [0, 1]$ is compact, $f$ is uniformly continuous. So by standard argument in point-set topology, $\m(\Phi, r)=\sup_{p\in M, x\in[0, 1]}f(r, p, x)\rightarrow 0$ when $r\rightarrow 0$, as $f(0, p, x)=\|\Si_{x}\|(\{p\})=0$.
\end{proof}

Given $\Si\in\mS$, we can define a mapping into $\big(\Z_{n}(M^{n+1}), \{0\}\big)$,
$$\Phi^{\Si}: [0, 1]\rightarrow\big(\Z_{n}(M^{n+1}), \{0\}\big)$$
as follows:
\begin{itemize}
\vspace{-5pt}
\addtolength{\itemsep}{-0.7em}
\item When $\Si\in\mS_{+}$, $\Phi^{\Si}(x)=[[\partial\Om_{2x-1}]]$ for $x\in[0, 1]$, where $\{\Om_{t}\}_{t\in[-1, 1]}$ is the family of open sets of $M$ in Definition \ref{definition of sweepout} corresponding to the sweepout $\{\Si_{t}\}_{t\in[-1, 1]}$ of $\Si$ constructed in Proposition \ref{existence of good sweepout};
\item When $\Si\in\mS_{-}$, $\Phi^{\Si}(x)=[[\partial\Om_{x}]]$ for $x\in[0, 1]$, where $\{\Om_{t}\}_{t\in[0, 1]}$ is the family of open sets of $M$ in Definition \ref{definition of sweepout} corresponding to the family $\{\Si_{t}\}_{t\in[0, 1]}$ of $\Si$ constructed in Proposition \ref{existence of good sweepout2}.
\end{itemize}
Then as a corollary of Proposition \ref{existence of good sweepout},  Proposition \ref{existence of good sweepout2} and Proposition \ref{continuity and no mass concentration}, we have,
\begin{corollary}\label{generating min-max family in flat topology}
$\Phi^{\Si}: [0, 1]\rightarrow\big(\Z_{n}(M^{n+1}), \{0\}\big)$ is continuous under the flat topology, and
\begin{itemize}
\vspace{-5pt}
\addtolength{\itemsep}{-0.7em}
\item[$(a)$] $\sup_{x\in[0, 1]}\M\big(\Phi^{\Si}(x)\big)=V(\Si)$ if $\Si\in\mS_{+}$;
\item[$(b)$] $\sup_{x\in[0, 1]}\M\big(\Phi^{\Si}(x)\big)=2V(\Si)$ if $\Si\in\mS_{-}$;
\item[$(c)$] $\m(\Phi^{\Si}, r)\rightarrow 0$, when $r\rightarrow 0$.
\end{itemize}
\end{corollary}
\begin{proof}
In the case $\Si\in\mS_{+}$, our conclusions are a direct consequence of Proposition \ref{continuity and no mass concentration}, as $\Phi^{\Si}$ satisfies the conditions there.

If $\Si\in\mS_{-}$, all the conclusions are true by Proposition \ref{existence of good sweepout2} and the proof of Proposition \ref{continuity and no mass concentration}, except that we need to check $(c)$. Using notions in Proposition \ref{existence of good sweepout2}, let $\ti{M}$ and $\ti{\Si}$ be the double cover of $M$ and $\Si$ respectively. Let $\ti{\Phi}^{\ti{\Si}}$ be the mapping corresponding to $\ti{\Si}$ in $\ti{M}$, then it is easy to see that $\m(\Phi^{\Si}, r)\leq 2\m(\ti{\Phi}^{\ti{\Si}}, r)$, hence we finish the proof by using the first case.
\end{proof}

\subsection{Discretize the min-max family}\label{Discretize the min-max family}

Now we will discretize the continuous family $\Phi^{\Si}$ to a $(1, \M)$-homotopy sequence as in Definition \ref{(1, M) homotopy sequence}. The idea originates from Pitts in \cite[\S 3.7 and \S 3.8]{P}. Marques and Neves first gave a complete statement in \cite[\S 13]{MN2} on generating $(m, \M)$-homotopy sequence into the space $\Z_{2}(M^{3})$ of integral two cycles in three manifold from a given min-max family continuous under the flat norm topology. Their proof never used any special feature for the special dimensions, so Theorem 13.1 in \cite{MN2} is still true to generate $(m, \M)$-homotopy sequence into $\Z_{n}(M^{n+1})$ from any continuous family under flat topology. While they used a contradiction argument, for the purpose of the proof of Theorem \ref{identification of (1, M) homotopy class}, we will give a modified direct discretization process based on ideas in \cite{P}\cite{MN2}. Our main result is an adaption of Theorem 13.1 in \cite{MN2}.
 
\begin{theorem}\label{generating (1, M) homotopy sequence}
Given a continuous mapping $\Phi: [0, 1]\rightarrow\big(\Z_{n}(M^{n+1}, \F), \{0\}\big)$, with
$$\sup_{x\in[0, 1]}\M(\Phi(x))<\infty, \textrm{ and } \lim_{r\rightarrow 0}\m(\Phi, r)=0,$$
there exists a $(1, \M)$ homotopy sequence
$$\phi_{i}: I(1, k_{i})_{0}\rightarrow \big(\Z_{n}(M^{n+1}, \M), \{0\}\big),$$
and a sequence
$$\psi_{i}: I(1, k_{i})_{0}\times I(1, k_{i})_{0}\rightarrow\Z_{n}(M^{n+1}, \M),$$
with $k_{i}<k_{i+1}$, and $\{\de_{i}\}_{i\in\N}$ with $\de_{i}>0$, $\de_{i}\rightarrow 0$, and $\{l_{i}\}_{i\in\N}$, $l_{i}\in\N$ with $l_{i}\rightarrow\infty$, such that $\psi_{i}([0], \cdot)=\phi_{i}$, $\psi_{i}([1], \cdot)=\phi_{i+1}|_{I(1, k_{i})_{0}}$, and
\begin{itemize}
\vspace{-5pt}
\setlength{\itemindent}{1em}
\addtolength{\itemsep}{-0.7em}
\item[$(1)$] $\M\big(\phi_{i}(x)\big)\leq \sup\big\{\M\big(\Phi(y)\big):\ x, y\in\al, \textrm{ for some 1-cell }\al\in I(1, l_{i})\big\}+\de_{i}$, hence
\begin{equation}
\bL(\{\phi_{i}\}_{i\in\N})\leq\sup_{x\in[0, 1]}\big(\Phi(x)\big);
\end{equation}
\item[$(2)$] $\f(\psi_{i})<\de_{i}$;
\item[$(3)$] $\sup\big\{\F\big(\psi_{i}(y, x)-\Phi(x):\ y\in I(1, k_{i})_{0}, x\in I(1, k_{i})_{0}\big)\big\}<\de_{i}$.
\end{itemize}
\end{theorem}

Before giving the proof, we first give a result which is a variation of \cite[Lemma 3.8]{P} and \cite[ Proposition 13.3]{MN2}. For completeness and for the purpose of application to the proof of Theorem \ref{identification of (1, M) homotopy class}, we will give a slightly modified sketchy proof. Denote $\B^{\F}_{\ep}(S)$ to be a ball of radius $\ep$ centered at $S$ in $\Z_{n}(M^{n+1}, \F)$.

\begin{lemma}\label{extension lemma}
Given $\de$, $r$, $L$ positive real numbers, and $T\in\Z_{n}(M^{n+1})\cap\{S: \M(S)\leq 2L\}$, there exists  $0<\ep=\ep(T, \de, r, L)<\de$, and $k=k(T, \de, r, L)\in\N$, such that whenever $S\in\B^{\F}_{\ep}(T)\cap\{S:\ \M(S)\leq 2L\}$, and $\m(S, r)<\frac{\de}{4}$, there exists a mapping $\ti{\phi}: I(1, k)_{0}\rightarrow\B^{\F}_{\ep}(T)$, satisfying
\begin{itemize}
\vspace{-5pt}
\setlength{\itemindent}{1em}
\addtolength{\itemsep}{-0.7em}
\item[$(i)$] $\ti{\phi}([0])=S$, $\ti{\phi}([1])=T$;
\item[$(ii)$] $\f(\ti{\phi})\leq\de$;
\item[$(iii)$] $\sup_{x\in I(1, k)_{0}}\ti{\phi}([x])\leq\M(S)+\de$.
\end{itemize}
\end{lemma}
\begin{proof}
By \cite[Corollary 1.14]{A1}, there exists $\ep_{M}>0$, such that if $\ep<\ep_{M}$, there exists $Q\in\bI_{n+1}(M^{n+1})$, such that
$$\partial Q=S-T,\quad \M(Q)=\F(S-T)<\ep.$$

We claim that there exists $\ep=\ep(T, \de, r, L)>0$ small enough and $v=v(T, \de, r, L)\in\N$ large enough, such that for any $S\in\B^{\F}_{\ep}(T)\cap\{S:\ \M(S)\leq 2L\}$, there exists a finite collection of disjoint balls $\{B_{r_{i}}(p_{i})\}_{i=1}^{v}$ with $r_{i}<r$, satisfying:
\begin{itemize}
\vspace{-5pt}
\addtolength{\itemsep}{-0.7em}
\item \begin{equation}\label{extension lemma equation 1}
\|S\|\big(B_{r_{i}}(p_{i})\big)\leq\frac{\de}{4},\ \|S\|\big(M\setminus \cup_{i=1}^{v}B_{r_{i}}(p_{i})\big)\leq\frac{\de}{4};
\end{equation}
\item \begin{equation}\label{extension lemma equation 2}
\|T\|\big(B_{r_{i}}(p_{i})\big)\leq\frac{\de}{3},\ \|T\|\big(M\setminus \cup_{i=1}^{v}B_{r_{i}}(p_{i})\big)\leq\frac{\de}{3};
\end{equation}
\item \begin{equation}\label{extension lemma equation 3}
(\|T\|-\|S\|)(B_{r_{i}}(p_{i}))\leq\frac{\de}{2v},\ (\|T\|-\|S\|)(M\setminus \cup_{i=1}^{v}B_{r_{i}}(p_{i}))\leq\frac{\de}{2v}.
\end{equation}
\item Denoting $d_{i}(x)=d(x, p_{i})$, the slice\footnote{See \cite[\S 28]{S} for definition of slices.} $\lan Q, d_{i}, r_{i}\ran\in\bI_{n}(M^{n+1})$, and
\begin{equation}\label{extension lemma equation 4}
\lan Q, d_{i}, r_{i}\ran=\partial\big(Q\lc B_{r_{i}}(p_{i})\big)-(\partial Q)\lc B_{r_{i}}(p_{i})=\partial\big(Q\lc B_{r_{i}}(p_{i})\big)-(S-T)\lc B_{r_{i}}(p_{i});
\end{equation}
\item \begin{equation}
\sum_{i=1}^{v}\M\big(\lan Q, d_{i}, r_{i}\ran\big)<\frac{\de}{2}.
\end{equation}
\end{itemize}
This claim follows from a contradiction argument. If it is not true, then there is a sequence $\ep_{j}\rightarrow 0$, and $S_{j}\in\B^{\F}_{\ep_{i}}(T)\cap\{S:\ \M(S)\leq 2L\}$, such that there exists no finite collection of disjoint balls satisfying the above properties. Then $\lim_{j\rightarrow\infty}S_{j}=T$, and weak compactness of varifolds with bounded mass implies that $\lim_{j\rightarrow\infty}|S_{j}|=V\in\V_{n}(M^{n+1})$ for some subsequence. Using the arguments in the proof of \cite[Lemma 13.4]{MN2} and \cite[Lemma 3.8]{P}, we can construct finite collection of disjoint balls satisfying the above requirement for each $S_{j}$ when $j$ is large enough, hence a contradiction. Notice that the condition $\m(S, r)<\frac{\de}{4}$ is essentially used to find the radius of the balls (see Lemma 13.4 in \cite{MN2} for details).

Define the map $\ti{\phi}: I(1, k)_{0}\rightarrow\Z_{n}(M^{n+1})$, with $k=N$, where we write $v=3^{N}-1$ for some $N\in\N$,  as follows:
\begin{equation}\label{extension map}
\begin{split}
& \ti{\phi}([\frac{i}{3^{N}}])=S-\sum_{a=1}^{i}\partial\big(Q\lc B_{r_{a}}(p_{a})\big),\ 0\leq i\leq 3^{N}-1;\\
& \ti{\phi}([1])=T.
\end{split}
\end{equation}
By arguments similar to \cite[Lemma 13.4]{MN2}, we can check that $\ti{\phi}(I(1, k)_{0})\subset\B^{\F}_{\ep}(T)$, and get the properties $(i)(ii)(iii)$ listed in the lemma using (\ref{extension lemma equation 1})(\ref{extension lemma equation 2})(\ref{extension lemma equation 3})(\ref{extension lemma equation 4}).
\end{proof}
\begin{remark}
In the proof of \cite[Lemma 13.4]{MN2} and \cite[Lemma 3.8]{P}, they used contradiction arguments to get the discretized maps, while we use contradiction arguments to get the good collection of balls.
\end{remark}

Now let us sketch the proof of Theorem \ref{generating (1, M) homotopy sequence}. Since the idea is the same as \cite[Lemma 13.1]{MN2}, we will mainly point out the ingredients which we will use in the following.
\begin{proof}
(of Theorem \ref{generating (1, M) homotopy sequence})
Fix a small $\de>0$. Let $L=\sup_{x\in[0, 1]}\M\big(\Phi(x)\big)$, and find $r>0$, such that $\m(\Phi, r)<\frac{\de}{4}$. By the compactness of $\Z_{n}(M^{n+1})\cap\{S:\ \M(S)\leq 2L\}$ under flat norm topology, we can find a finite cover of $\Z_{n}(M^{n+1})\cap\{S:\ \M(S)\leq 2L\}$, containing $\{\B^{\F}_{\ep_{i}}(T_{i}):\ i=1, \cdots, N\}$, with
$$T_{i}\in \Z_{n}(M^{n+1})\cap\{S:\ \M(S)\leq 2L\},\ \ep_{i}=\frac{\ep(T_{i}, \de, r, L)}{8},$$
where $\ep(T_{i}, \de, r, L)$ and $k_{i}=k(T_{i}, \de, r, L)$ are given by Lemma \ref{extension lemma}.

By the continuity of $\Phi$, we can take $j_{\de}\in\N$ large enough, such that for any $1$-cell $\al\in I(1, j_{\de})$, $\Phi(\al_{0})\subset\B^{\F}_{\ep_{i(\al)}}(T_{i(\al)})$ for some $i(\al)$ depending on $\al$. 

Now fix a $1$-cell $\al\in I(1, j_{\de})$, with $\al=[t^{1}_{\al}, t^{2}_{\al}]$. Then $\Phi(t^{l}_{\al})\in\B^{\F}_{\ep_{i(\al)}}(T_{i(\al)})$, and $\m\big(\Phi(t^{l}_{\al}), r\big)<\frac{\de}{4}$, for $l=1, 2$. By Lemma \ref{extension lemma}, there exists $\ti{\phi}^{l}_{\al}: I(1, k_{i})_{0}\rightarrow\B^{\F}_{\ep_{i(\al)}}(T_{i(\al)})$, such that: $\ti{\phi}^{l}_{\al}([0])=\Phi(t^{l}_{\al})$, $\ti{\phi}^{l}_{\al}([1])=T_{i(\al)}$, $\f(\ti{\phi}^{l}_{\al})\leq\de$, and $\sup\{\M\big(\ti{\phi}^{l}_{\al}(x)\big):\ x\in I(1, k_{i})_{0}\}\leq\M\big(\Phi(t^{l}_{\al})\big)+\de$.

By identifying $\al$ with $[0, 1]$, we can define $\ti{\phi}_{\al}: \al(k_{i}+1)_{0}\rightarrow\B^{\F}_{\ep_{i(\al)}}(T_{i(\al)})$ as follows:
\begin{equation}\label{construction of phi-al}
\left. \ti{\phi}_{\al}([\frac{j}{3^{k_{i}+1}}])= \Bigg\{ \begin{array}{ll}
\ti{\phi}^{1}_{\al}([\frac{j}{3^{k_{i}+1}}]), \quad \textrm{ if $j=0, \cdots, 3^{k_{i}}$};\\
T_{i(\al)}, \quad \textrm{ if $j=3^{k_{i}}, \cdots, 2\cdot 3^{k_{i}}$};\\
\ti{\phi}^{2}_{\al}([\frac{\cdot 3^{k_{i}+1}-j}{3^{k_{i}+1}}]), \quad \textrm{ if $j=2\cdot 3^{k_{i}}, \cdots, 3^{k_{i}+1}$}.
\end{array}\right. 
\end{equation}
Then for $k_{\de}=\max_{i=1}^{N}\{k_{i}\}$, we can define: $\phi_{\de}: I(1, j_{\de}+k_{\de}+1)_{0}\rightarrow \Z_{n}(M^{n+1})$ as follows:
\begin{equation}\label{construction of phi-de}
\phi_{\de}|_{\al(k_{\de}+1)_{0}}=\ti{\phi}_{\al}\circ n(k_{\de}+1, k_{i}+1),\ \textrm{ for any $1$-cell $\al\in I(1, j_{\de})$},
\end{equation}
where $n(i, j)$ is as in $(6)$ of Definition \ref{cell complex}. From Lemma \ref{extension lemma}, we know that: $\phi_{\de}|_{I(1, j_{\de})_{0}}=\Phi|_{I(1, j_{\de})_{0}}$, $\f(\phi_{\de})\leq\sup_{\al\in I(1, j_{\de})_{1}}\f(\ti{\phi}_{\al})\leq\de$, and
$$\M\big(\phi_{\de}(x)\big)\leq \sup\{\M\big(\Phi(y)\big): y, x\in\al, \textrm { for some $1$-cell $\al\in I(1, j_{\de})$}\}.$$

Now take a sequence of positive numbers $\{\de_{i}\}_{i\in\N}$, with $\de_{i}\rightarrow 0$ as $i\rightarrow\infty$. Construct $\phi_{i}=\phi_{\de_{i}}: I(1, j_{\de_{i}}+k_{\de_{i}}+1)_{0}\rightarrow\Z_{n}(M^{n+1})$ as above. By taking a subsequence, we can construct the sequence of $1$-homotopy $\{\psi_{i}\}_{i\in\N}$ as in the second part of \cite[Theorem 13.1]{MN2}. The properties $(1)(2)(3)$ listed in the theorem follow from the arguments there.
\end{proof}

In order to prove the final result, we need to show that the $(1, \M)$-homotopy sequences of mappings into $\big(\Z_{n}(M^{n+1}), \{0\}\big)$, which are constructed above from the mapping $\Phi^{\Si}$ in Corollary \ref{generating min-max family in flat topology} for any $\Si\in\mS$, belong to the same homotopy class in $\pi^{\#}_{1}\big(\Z_{n}(M^{n+1}), \{0\}\big)$. Similar issue was considered in the proof of \cite[Theorem 8.4]{MN2}. However, they only need to show that their sequence is non-trivial, while we need to identify all our sequences. First we have the following theorem.

\begin{theorem}\label{identification of (1, M) homotopy class}
Given $\Phi$ as in Theorem \ref{generating (1, M) homotopy sequence}, and $\{\phi_{i}\}_{i\in\N}$ the corresponding $(1, \M)$-homotopy sequence obtained by Theorem \ref{generating (1, M) homotopy sequence}. Assume that $\Phi(x)=[[\partial\Om_{x}]]$, $x\in[0, 1]$ where $\{\Om_{t}\}_{t\in[0, 1]}$ is a family of open sets satisfying $(sw2)(sw3)$ in Definition \ref{definition of sweepout}. If $F: \pi_{1}^{\#}\big(\Z_{n}(M^{n+1}), \{0\}\big)\rightarrow H_{n+1}(M^{n+1}, \mZ)$ is the isomorphism given by Almgren in Section 3.2 in \cite{A1}, then
$$F\big([\{\phi_{i}\}_{i\in\N}]\big)=[[M]],$$
where $[[M]]$ is the fundamental class of $M$. 
\end{theorem}
\begin{proof}
We will directly cite the notions in the proof of Theorem \ref{generating (1, M) homotopy sequence}. First we review the definition of $F$ given in \cite[\S 3.2]{A1}. Fix an $i$ large enough, with $\de_{i}$ small enough, and we will omit the sub-index $i$ in the following. Take $\phi_{\de}=\phi_{\de_{i}}: I(1, j_{\de}+k_{\de}+1)_{0}\rightarrow \Z_{n}(M^{n+1})$ constructed in Theorem \ref{generating (1, M) homotopy sequence}. For any $1$-cell $\be\in I(j_{\de}+k_{\de}+1)$, with $\be=[t^{1}_{\be}, t^{2}_{\be}]$, $\F\big(\phi_{\de}(t^{1}_{\be}), \phi_{\de}(t^{2}_{\be})\big)\leq\M\big(\phi_{\de}(t^{1}_{\be}), \phi_{\de}(t^{2}_{\be})\big)\leq\f(\phi_{\de})\leq\de$. By \cite[Corollary 1.14]{A1}, there exists an isoperimetric choice $Q_{\be}\in\bI_{n+1}(M^{n+1})$, with $\M(Q_{\be})=\F\big(\phi_{\de}(t^{1}_{\be}), \phi_{\de}(t^{2}_{\be})\big)$, and
$$\partial Q_{\be}=\phi_{\de}(\partial\be)=\phi_{\de}(t^{2}_{\be})-\phi_{\de}(t^{1}_{\be}).$$
Then $F$ is defined in \cite[\S 3.2]{A1} as:
\begin{equation}\label{definition of F}
F\big([\{\phi_{i}\}_{i\in\N}]\big)=\sum_{\be\in I(1, j_{\de}+k_{\de}+1)_{1}}[[Q_{\be}]],
\end{equation}
where the right hand side is a $n+1$ dimensional integral cycle as $\phi_{\de}([0])=\phi_{\de}([1])=0$, which hence represents a $n+1$ dimensional integral homology class.

For any $1$-cell $\al\in I(1, j_{\de})$, we denote
\begin{equation}\label{definition of tiF}
\ti{F}(\al, \phi_{\de})=\sum_{\be\in\al(k_{\de}+1)_{1}}[[Q_{\be}]].
\end{equation}

Now let us identify the right hand side of (\ref{definition of F}) with $[[M]]$ using our construction. Let $\{\Om_{t}\}_{t\in[0, 1]}$ be the defining open sets of $\Phi$. From the construction of $\phi_{\de}$, we know $\phi_{\de}|_{I(1, j_{\de})_{0}}=\Phi|_{I(1, j_{\de})_{0}}$, so
$$\phi_{\de}([\frac{j}{3^{j_{\de}}}])=\Phi(\frac{j}{3^{j_{\de}}})=[[\partial\Om_{\frac{j}{3^{j_{\de}}}}]],$$
by the definition of $\Phi$.
\begin{claim}\label{image of tiF}
For the $1$-cell $\al_{j}=[\frac{j}{3^{j_{\de}}}, \frac{j+1}{3^{j_{\de}}}]$,
$$\ti{F}(\al_{j}, \phi_{\de})=[[\Om_{\frac{j+1}{3^{j_{\de}}}}]]-[[\Om_{\frac{j}{3^{j_{\de}}}}]].$$
\end{claim}
Hence
$$F\big([\{\phi_{i}\}_{i\in\N}]\big)=\sum_{\al\in I(1, j_{\de})_{1}}\ti{F}(\al, \phi_{\de})=\sum_{j=0}^{3^{j_{\de}}-1}[[\Om_{\frac{j+1}{3^{j_{\de}}}}-\Om_{\frac{j}{3^{j_{\de}}}}]]=[[\Om_{1}]]=[[M]].$$

Let us go back to check the claim. Take $\al=\al_{j}=[\frac{j}{3^{j_{\de}}}, \frac{j+1}{3^{j_{\de}}}]$. Since $\phi_{\de}|_{\al(k_{\de}+1)_{0}}=\ti{\phi}_{\al}\circ n(k_{\de}+1, k_{i(\al)}+1)$ by (\ref{construction of phi-de}), it is easy to see that
$$\ti{F}(\al, \phi_{\de})=\ti{F}(\al, \ti{\phi}_{\al})=\sum_{\be\in \al(k_{i(\al)}+1)_{1}}[[Q_{\be}]].$$
By identifying $\al=[0, 1]$, the mapping $\ti{\phi}_{\al}: I(k_{i(\al)}+1)_{0}\rightarrow\Z_{n}(M^{n+1})$ is a combination of three parts by (\ref{construction of phi-al}), especially $\ti{\phi}_{\al}|_{[\frac{1}{3}, \frac{2}{3}](k_{i(\al)})_{0}}\equiv T_{i(\al)}$, hence
$$\ti{F}(\al, \ti{\phi}_{\al})=\ti{F}(\ti{\phi}^{1}_{\al})+\ti{F}(\ti{\phi}^{2}_{\al}).$$
Take $\ti{\phi}^{1}_{\al}: I(1, k_{i(\al)})_{0}\rightarrow\Z_{n}(M^{n+1})$ for example. From the construction, there exists an isoperimetric choice $Q_{\al, 1}\in\bI_{n+1}(M^{n+1})$, such that $\partial Q_{\al, 1}=\Phi([\frac{j}{3^{j_{\de}}}])-T_{i(\al)}=[[\partial\Om_{\frac{j}{3^{j_{\de}}}}]]-T_{i(\al)}$, and $\M(Q_{\al, 1})\leq\F\big(\Phi([\frac{j}{3^{j_{\de}}}]), T_{i(\al)}\big)\leq\ep_{\al}<\de$. Then from (\ref{extension map}), we have
$$\ti{\phi}^{1}_{\al}([\frac{h}{3^{k_{i(\al)}}}])=[[\partial\Om_{\frac{j}{3^{j_{\de}}}}]]-\sum_{a=1}^{h}\partial\big(Q_{\al, 1}\lc B_{r_{a}}(p_{a})\big),\ 1\leq h\leq 3^{k_{i(\al)}}-1;\quad \ti{\phi}^{1}_{\al}([1])=T_{i(\al)}.$$
Take the isoperimetric choice $Q_{\al, 1, h}\in\bI_{n+1}(M^{n+1})$, such that
$$\partial Q_{\al, 1, h}=\ti{\phi}_{\al}([\frac{h}{3^{k_{i(\al)}}}])-\ti{\phi}_{\al}([\frac{h-1}{3^{k_{i(\al)}}}])=-\partial\big(Q_{\al, 1}\lc B_{r_{h}}(p_{h})\big),\ 1\leq h\leq 3^{k_{i(\al)}}-1;$$
$$\partial Q_{\al, 1, 3^{k_{i(\al)}}}=T_{i(\al)}-\ti{\phi}_{\al}([\frac{3^{k_{i(\al)}}-1}{3^{k_{i(\al)}}}])=-\partial\bigg(Q_{\al, 1}\lc \big(M\setminus\cup_{h=1}^{v}B_{r_{h}}(p_{h})\big)\bigg).$$
So
$$\sum_{h=1}^{3^{k_{i(\al)}}}\partial Q_{\al, 1, h}=-\partial Q_{\al, 1}=T_{i(\al)}-[[\partial\Om_{\frac{j}{3^{j_{\de}}}}]],$$
and from the definition of isoperimetric choice (see \cite[Corollary 1.14]{A1}),
\begin{displaymath}
\begin{split}
\sum_{h=1}^{3^{k_{i(\al)}}}\M(Q_{\al, 1, h}) & \leq\sum_{h=1}^{3^{k_{i(\al)}}-1}\M\big(Q_{\al, 1}\lc B_{r_{h}}(p_{h})\big)+\M\bigg(Q_{\al, 1}\lc\big(M\setminus\cup_{h=1}^{v}B_{r_{h}}(p_{h})\big)\bigg)\\
                                                                 & =\M(Q_{\al, 1})<\de.
\end{split}
\end{displaymath}
Similar results hold for $\ti{\phi}^{2}_{\al}$, so
$$\ti{F}(\al, \ti{\phi}_{\al})=\ti{F}(\ti{\phi}^{1}_{\al})+\ti{F}(\ti{\phi}^{2}_{\al})=\sum_{h=1}^{3^{k_{i(\al)}}}[[Q_{\al, 1, h}]]+\sum_{h=1}^{3^{k_{i(\al)}}}[[Q_{\al, 2, h}]],$$
with $\M\big(\ti{F}(\al, \ti{\phi}_{\al})\big)<2\de$, and
$$\partial\big(\ti{F}(\al, \ti{\phi}_{\al})\big)=T_{i(\al)}-[[\partial\Om_{\frac{j}{3^{j_{\de}}}}]]+[[\partial\Om_{\frac{j+1}{3^{j_{\de}}}}]]-T_{i(\al)}=\partial[[\Om_{\frac{j+1}{3^{j_{\de}}}}-\Om_{\frac{j}{3^{j_{\de}}}}]].$$
Hence $\partial\big(\ti{F}(\al, \ti{\phi}_{\al})-[[\Om_{\frac{j+1}{3^{j_{\de}}}}-\Om_{\frac{j}{3^{j_{\de}}}}]]\big)=0$, so using the Constancy Theorem (\cite[Theorem 26.27]{S}), we know that $\ti{F}(\al, \ti{\phi}_{\al})-[[\Om_{\frac{j+1}{3^{j_{\de}}}}-\Om_{\frac{j}{3^{j_{\de}}}}]]=k[[M]]$ for some $k\in\mZ$. Since that $\M\big(\ti{F}(\al, \ti{\phi}_{\al})-[[\Om_{\frac{j+1}{3^{j_{\de}}}}-\Om_{\frac{j}{3^{j_{\de}}}}]]\big)\leq 2\de+\textrm{Volume}\big(\Om_{\frac{j+1}{3^{j_{\de}}}}\triangle\Om_{\frac{j}{3^{j_{\de}}}}\big)$ is small enough for large $j_{\de}$, we know that $k=0$, hence $\ti{F}(\al, \ti{\phi}_{\al})=[[\Om_{\frac{j+1}{3^{j_{\de}}}}-\Om_{\frac{j}{3^{j_{\de}}}}]]$, so we proved the claim and finished the proof.
\end{proof}

Now we can combine all the results above to get discretized sequences and show that they all lie in the same homotopy class. Given $\Si\in\mS$, let $\Phi^{\Si}: [0, 1]\rightarrow\big(\Z_{n}(M^{n+1}), \{0\}\big)$ be the mapping given in Corollary \ref{generating min-max family in flat topology}. Then we can apply Theorem \ref{generating (1, M) homotopy sequence} to get a $(1, \M)$-homotopy sequence $\{\phi^{\Si}_{i}\}_{i\in\N}$ into $\big(\Z_{n}(M^{n+1}, \F), \{0\}\big)$. Clearly
\begin{equation}\label{min-max bound}
\left. \bL(\{\phi^{\Si}_{i}\}_{i\in\N})\leq \Big\{ \begin{array}{ll}
V(\Si), \quad \textrm{ if $\Si\in\mS_{+}$};\\
2V(\Si), \quad \textrm{ if $\Si\in\mS_{-}$}.
\end{array}\right. 
\end{equation}
Then a direct corollary of Theorem \ref{identification of (1, M) homotopy class} is,
\begin{corollary}\label{corollary of identification of (1, M) homotopy class}
$[\{\phi^{\Si}_{i}\}_{i\in\N}]=F^{-1}([[M]])\in\pi_{1}^{\#}\big(\Z(M^{n+1}), \{0\}\big)$, for any $\Si\in\mS$.
\end{corollary}


\section{Orientation and multiplicity}\label{orientation and multiplicity}

In this section, we will discuss the orientation and multiplicity of the min-max hypersurface. In Theorem \ref{AP min-max theorem}, the stationary varifold $\Si$ is an integral multiple of some smooth minimal hypersurface (denoted still as $\Si$). The fact that $\Si$ lies in the critical set $C(S)$ of some critical sequence $S$ implies that $\Si$ is a varifold limit of a sequence of integral cycles $\{\phi_{i_{j}}(x_{j})\}_{j\in\N}\subset\Z_{n}(M^{n+1})$. The weak compactness implies that $\{\phi_{i_{j}}(x_{j})\}_{j\in\N}$ converge to a limit integral current up to a subsequence, which is then supported on $\Si$. It has been conjectured that $\Si$ should have some orientation structures by comparing the varifold limit and current limit. Hence it motivates us to prove the following result. In fact, this result holds for all Riemannian manifolds.

\begin{proposition}\label{orientation and multiplicity result}
Let $\Si$ be the stationary varifold in Theorem \ref{AP min-max theorem}, with $\Si=\cup_{i=1}^{l}k_{i}[\Si_{i}]$, where $\{\Si_{i}\}$ is a disjoint collection of smooth connected closed embedded minimal hypersurfaces with multiplicity $k_{i}\in\N$. If $\Si_{i}$ is non-orientable, then the multiplicity $k_{i}$ must be an even number.
\end{proposition}
\begin{remark}
This is a characterization of the orientation structure of the min-max hypersurface. When a connected component of $\Si$ is orientable, it naturally represents an integral cycle. While a connected component of $\Si$ is non-orientable, an even multiplicity of it also represents an integral cycle---a zero integral cycle.
\end{remark}

Let us first introduce our strategy for proving this result. Two key ingredients will be used. The first key ingredient is an important general property of the min-max varifold called the ``almost minimizing" property \cite[\S 3.1]{P}. The almost minimizing property implies that the min-max varifold has a local replacement, which is a varifold limit of a sequence of integral cycles, that are locally mass minimizing in the region where it replaces the original min-max varifold. In the case of co-dimension one theory, Pitts \cite[Chap 7]{P} essentially showed that the local replacement, which is regular in the replacement region, coincides with the original min-max varifold locally. Hence it implies the regularity of the original min-max varifold. Here we will first show that good local replacements coincide with the original min-max varifold globally by exploring Pitts's idea. Then we need the second key ingredient, which is a convergence result by B. White \cite{W}. In fact, for a sequence of integral currents where all the associated varifolds have locally bounded first variations, White showed that the varifold limit and the current limit of this sequence can differ at most by an even multiple of some integral varifold. By applying White's result to the sequence of integral currents that converges to the local replacement, we can show that the replacement, the same as the original min-max hypersurface, must have even multiplicity when it is non-orientable.\\

First let us introduce some concepts related to the ``almost minimizing" property (see \cite[\S 3.1]{P}). Let $M^{n+1}$ be an arbitrary Riemannian manifold, and $U$ a bounded open subset of $M$. We use $B(p, r)$ and $A(p, s, r)=B(p, r)\setminus\overline{B(p, s)}$ to denote the open ball and open annulus in $M$. Given $k\in\N$ and $1\leq k\leq n$.
\begin{definition}\label{definition for almost minimizing}
Given $\ep>0$ and $\de>0$, $\mA_{k}(U, \ep, \de)$ is the set of integral cycles $T\in\Z_{k}(M)$, such that: if $T=T_{0}, T_{1}, \cdots, T_{m}\in\Z_{k}(M)$ with
$$\textrm{spt}(T-T_{i})\subset U,\ \F(T_{i}, T_{i-1})\leq \de,\  \M(T_{i})\leq \M_{T}+\de,\ \textrm{for } i=1, \cdots, m,$$
then $\M(T_{m})\geq \M(T)-\ep$.

A rectifiable varifold $V\in\V_{k}(M)$ is called \emph{almost minimizing} in $U$, if for any $\ep>0$, there exists a $\de>0$ and $T\in\mA_{k}(U, \ep, \de)$, such that $\mF(V, |T|)<\ep$\footnote{This $\mF$ is the $\mF$-metric for varifold defined in \cite[page 66]{P}, which defines the varifold weak topology}.
\end{definition}
\begin{remark}\label{remark of almost minimizing}
In the original work of Pitts (see \cite[\S 3.1]{P}), the definition of $\mA_{k}(U, \cdot, \cdot)$ uses comparison currents $T\in\Z_{k}(M, M\setminus U)$, i.e. integral currents with boundary outside $U$, and the almost minimizing varifold is defined to be approximated by $\mA_{k}(U, \cdot, \cdot)$ under $\mF_{U}$ norm. Our definition is stronger and implies Pitts's definition in \cite[\S 3.1]{P}, so we can use all the regularity results in \cite{P}. Moreover, the min-max varifold appearing in \cite[Theorem 4.10]{P} does satisfy our definition (in small annulli). In fact, the contradiction arguments (see Part 2 in the proof on \cite[Theorem 4.10, page 164]{P}) are made with respect to our definition of ``almost minimizing". The observation of this stronger version of ``almost minimizing" will enable us to gain global properties of the min-max hypersurface.
\end{remark}

Now we introduce the concept of local replacement. Let $M$ and $U$ be as above. We have the following result which is exactly \cite[Theorem 3.11]{P} adapted to our definition of ``almost minimizing".
\begin{theorem}\label{property of replacement}
Suppose $V\in\V_{k}(M)$ is almost minimizing in $U$, and $K$ is a compact subset of $U$. Then there is a nonempty set $\mR(V; U, K)\subset \V_{k}(M)$, such that any $V^{*}\in\mR(V; U, K)$ satisfies:
\begin{itemize}
\vspace{-5pt}
\setlength{\itemindent}{1em}
\addtolength{\itemsep}{-0.7em}
\item[$(1)$] $V^{*}\lc G_{k}(M\setminus K)\footnote{$G_{k}(\cdot)$ is the Grassmann manifold \cite[\S 2.1(12)]{P}.}=V\lc G_{k}(M\setminus K)$;
\item[$(2)$] $V^{*}$ is almost minimizing in $U$;
\item[$(3)$] $\|V^{*}\|(M)=\|V\|(M)$;
\item[$(4)$] $\forall\ \ep>0$, $\exists\ T\in\Z_{k}(M)$, such that $\mF(V^{*}, |T|)<\ep$, and $T\lc Z$ is locally mass minimizing with respect to $(Z, \emptyset)$ for all compact Lipschitz neighborhood retract $Z\subset\textrm{Int}(K)$;
\vspace{-5pt}
\end{itemize}
We will call such $V^{*}$ a replacement of $V$ in $K$.
\end{theorem}
\begin{remark}
The construction of $V^{*}$ is given in \cite[\S 3.10]{P}. The only difference here is property $(4)$. Due to our definition of almost minimizing, the approximation current $T$ can be chosen as an integral cycle rather in $\Z_{k}(M, M\setminus U)$, and the approximation can be made under $\mF$-norm rather than $\mF_{U}$-norm.
\end{remark}

Now let us cite some regularity results from \cite[Chap 7]{P} for the replacements of almost minimizing varifolds in the co-dimension one case.
\begin{lemma}\label{regularity for replacement}
\emph{(\cite[Corollary 7.7]{P})} Suppose $2\leq k\leq 6$, $M^{k+1}$ is a given Riemannian manifold, and $U$ is a bounded open subset of $M$. If $K$ is a compact subset of $U$, $V\in\V_{k}(M)$ is almost minimizing in $U$, and $V^{*}\in\mR(V; U, K)$, then $\textrm{spt}(\|V^{*}\|)\cap\textrm{Int}(K)$ is a $k$ dimensional smooth submanifold, which is stable in $\textrm{Int}(K)$.
\end{lemma}
\begin{remark}
The case $k=6$ is due to \cite[equation (7.4)]{SS}.
\end{remark}

Another useful result in \cite{P} is the following identification lemma.
\begin{lemma}\label{identification of replacement 1}
\emph{(\cite[Lemma 7.10]{P})} Let $k, M$ be as above. Given $p\in M$ and $r>0$ small enough. If $V\in\V_{k}(M)$ is almost minimizing in $B(p, 2r)$ and $\textrm{spt}(\|V\|)\cap A(p, \frac{r}{2}, r)$ is a smooth submanifold in $M$, then for $L^{1}$ almost all $\frac{r}{2}<s<r$, if $V^{*}\in\mR(V; B(p, r), \overline{B(p, s)})$, then $V^{*}\lc G_{k}\big(A(p, \frac{r}{2}, s)\big)=V\lc G_{k}\big(A(p, \frac{r}{2}, s)\big)$.
\end{lemma}
\begin{remark}
The case $k=6$ is again due to \cite[equation (7.40)]{SS}.
\end{remark}

Using the results above, we can show that good local replacement coincide with the min-max hypersurface globally.
\begin{lemma}\label{identification of replacement 2}
Given $k, M, p, r$ as above. Suppose $V\in\V_{k}(M)$ is almost minimizing in $B(p, 2r)\subset M$, and $\textrm{spt}(\|V\|)\cap B(p, 2r)$ is a smooth connected embedded minimal hypersurface. Then for $s\in[\frac{r}{2}, r]$ as in Lemma \ref{identification of replacement 1} with $\|V\|(\partial B(p, s))=0$ (which exists due to transversality), if $V^{*}\in\mR(V; B(p, r), \overline{B(p, s)})$, then $V^{*}=V$.
\end{lemma}
\begin{proof}
By Lemma \ref{identification of replacement 1}, $V^{*}\lc G_{k}\big(A(p, \frac{r}{2}, s)\big)=V\lc G_{k}\big(A(p, \frac{r}{2}, s)\big)$. By Lemma \ref{regularity for replacement}, $\textrm{spt}(\|V^{*}\|)\cap B(p, s)$ is a smooth embedded minimal hypersurface. As $\textrm{spt}(\|V\|)\cap B(p, 2r)$ is connected, the classical unique continuation for minimal hypersurface (c.f. \cite[Theorem 5.3]{DT}) implies that $\textrm{spt}(\|V\|)\cap B(p, s)\subset \textrm{spt}(\|V^{*}\|)\cap B(p, s)$. By $(1)(3)$ in Theorem \ref{property of replacement}, it is easy to see that $\|V^{*}\|(B(p, s))=\|V\|(B(p, s))$ and $\|V^{*}\|(\partial B(p, s))=0$. Hence $V^{*}\lc \overline{B(p, s)}=V\lc \overline{B(p, s)}$, so $V^{*}=V$.
\end{proof}

Finally we need the following convergence result by White \cite{W}.
\begin{theorem}\label{White convergence theorem}
\emph{(\cite[Theorem 1.2]{W})} Let $\{T_{i}\}_{i\in\N}$ and $\{V_{i}\}_{i\in\N}$ be sequences of integral currents and integral varifolds with $V_{i}=[T_{i}]$. If $V_{i}$ have locally bounded first variation, and if $\partial T_{i}$ converge to a limit current. Then for a subsequence $V_{i}$ converge to an integral varifold $V$ and $T_{i}$ converge to an integral current $T$, such that $V=[T]+2W$ for some integral varifold $W$.
\end{theorem}

Now we are ready to prove Proposition \ref{orientation and multiplicity result}.
\begin{proof}
(of Proposition \ref{orientation and multiplicity result}) By \cite[Theorem 4.10]{P} and Remark \ref{remark of almost minimizing}, for any $p\in M$, there exists $r_{p}>0$, such that $\Si$ is almost minimizing (in the sense of Definition \ref{definition for almost minimizing}) in $A(p, s, r_{p})$ for all $0<s<r_{p}$. Let $\Si_{1}$ be a non-orientable component of $\Si$. Hence we can take a point $p\in\Si_{1}$, and $r>0$ small enough, such that $\Si$ is almost minimizing in $B(p, 2r)$ (can choose $B(p, 2r)$ as a ball inside some open annulus $A(p^{\pr}, s, r_{p^{\pr}})$), and $\textrm{spt}(\|\Si\|)\cap B(p, 2r)=\textrm{spt}(\|\Si_{1}\|)\cap B(p, 2r)$ is diffeomorphic to a $n$-ball. Take $s\in[\frac{r}{2}, r]$ as in Lemma \ref{identification of replacement 2}, and $V^{*}\in\mR(V; B(p, r), \overline{B(p, s)})$, then $V^{*}=V$.

By $(4)$ in Theorem \ref{property of replacement}, there exists a sequence of integral cycles $\{T_{i}\}_{i\in\N}\subset\Z_{n}(M^{n+1})$, satisfying: $\lim_{i\rightarrow\infty}[T_{i}]=\Si$ as varifolds, and $T_{i}\lc B(p, s)$ is locally mass minimizing in $B(p, s)$. The co-dimension one regularity theory (c.f. \cite[Theorem 7.2]{P}\cite[Theorem 37.7]{S}) implies that $\textrm{spt}(T_{i})\lc B(p, s)$ are smooth embedded stable minimal hypersurfaces.

Since $\lim_{i\rightarrow\infty}[T_{i}]=\Si$ as varifolds, $\M(T_{i})$ are uniformly bounded, hence the weak compactness theorem for integral currents (\cite[Theorem 27.3]{S}) implies that a subsequence, still denoted by $\{T_{i}\}$, converges to some integral current $T_{0}\in\Z_{n}(M^{n+1})$, i.e. $\lim_{i\rightarrow\infty}T_{i}=T_{0}$. Since the associated Radon measure $\|T_{i}\|$ converges to $\|\Si\|$ weakly by the varifold convergence, we know that $T_{0}$ must have its support in $\cup_{i=1}^{l}\Si_{i}$, i.e. $\textrm{spt}(T_{0})\subset\cup_{i=1}^{l}\Si_{i}$. As an elementary fact, we have (see the proof in Appendix \ref{appendix}),
\begin{claim}\label{current in submanifold}
$T_{0}$ is an integral $n$-cycle in $\cup_{i=1}^{l}\Si_{i}$, i.e. $T_{0}\in\Z_{n}(\cup_{i=1}^{l}\Si_{i})$.
\end{claim}
By the Constancy Theorem \cite[Theorem 26.27]{S}\footnote{Here we can first find a finite covering of $\Si_{0}$, with each open set diffeomorphic to a Euclidean ball, and then apply the Constancy Theorem to each open set of the covering, and finally patch the results together.}, $T_{0}=\sum_{i=1}^{l}[[k^{\pr}_{i}\Si_{i}]]$, for some $k^{\pr}_{i}\in\mZ$. As $\Si_{1}$ in non-orientable, $k^{\pr}_{1}$ must be zero, or $k^{\pr}_{1}\Si_{1}$ could not represent an integral cycle. The lower semi-continuity of the mass implies that $|k^{\pr}_{i}|\leq k_{i}$, for $i=1,\cdots, l$.

Now let us focus on the ball $B(p, s)$. After possibly shrinking the radius, we can assume that $\partial(T_{i}\lc B(p, s))$ have uniformly bounded mass (by slicing theory \cite[Lemma 28.5]{S}), which hence converge to a limit current up to a subsequence. Clearly $[T_{i}]\lc B(p, s)$ have bounded first variation since they are represented by smooth stable minimal hypersurfaces. Then Theorem \ref{White convergence theorem} implies that $\Si\lc B(p, s)=k_{1}[\Si_{1}]\lc B(p, s)=[T_{0}\lc B(p, s)]+2W=2W$, for some integral varifold $W$. So $k_{1}$ is even.
\end{proof}


\section{Proof of the main result}\label{main result}

Now we are ready to prove the main results.

\begin{proof}(of Theorem \ref{main theorem1})
For any $\Si\in\mS$, take $\Phi^{\Si}$ as in Corollary \ref{generating min-max family in flat topology}, and let the corresponding $(1, \M)$-homotopy sequence be $S_{\Si}=\{\phi^{\Si}_{i}\}_{i\in\N}$. From Corollary \ref{corollary of identification of (1, M) homotopy class}, all $S_{\Si}$ lie in the same homotopy class $F^{-1}([[M]])$, which we denote by $\Pi_{M}$, then $\Pi_{M}$ is nontrivial by Theorem \ref{isomorphism}. We know from (\ref{min-max bound}) that,
$$\bL(\Pi_{M})\leq W_{M},$$
where $W_{M}$ is defined in (\ref{minimal volume of embedded minimal hypersurfaces}). Then we can apply the Almgren-Pitts Min-max Theorem \ref{AP min-max theorem}, so there exists a stationary integral varifold $\Si$, whose support is a closed smooth embedded minimal hypersurface $\Si_{0}$, such that $\bL(\Pi_{M})=\|\Si\|(M)$. Notice that $\Si_{0}$ must be connected by Theorem \ref{intersection}. Hence $\Si=k[\Si_{0}]$ for some $k\in\N$, $k\neq 0$. So
\begin{equation}\label{width control}
kV(\Si_{0})=\|\Si\|(M)=\bL(\Pi_{M})\leq W_{M},
\end{equation}
and from the definition (\ref{minimal volume of embedded minimal hypersurfaces}) of $W_{M}$,
\begin{itemize}
\vspace{-5pt}
\addtolength{\itemsep}{-0.7em}
\item If $\Si_{0}\in\mS_{+}$, orientable, then $k\leq 1$, hence $k=1$;
\item If $\Si_{0}\in\mS_{-}$, non-orientable, then $k\leq 2$, hence $k=1$ or $k=2$.
\end{itemize}

First let us see the case $\Si_{0}\in\mS_{-}$. By Proposition \ref{orientation and multiplicity result}, $k$ must be even, hence $k=2$. By (\ref{minimal volume of embedded minimal hypersurfaces}) and (\ref{width control}) $W_{M}\leq 2V(\Si_{0})\leq W_{M}$, which implies that $2V(\Si_{0})=W_{M}$. So we proved the case $(ii)$.

If $\Si_{0}\in\mS_{+}$, then by (\ref{minimal volume of embedded minimal hypersurfaces}) and (\ref{width control}) again $W_{M}\leq V(\Si_{0})\leq W_{M}$, which implies that $V(\Si_{0})=W_{M}$.
\begin{claim}\label{index one}
In this case, $\Si_{0}$ has index one.
\end{claim}
Let us check the claim now. As in the proof of Proposition \ref{existence of good sweepout}, there exists an eigenfunction $u_{1}$ of the Jacobi operator $L_{\Si_{0}}$, with $L_{\Si_{0}}u_{1}>0$ and $u_{1}>0$. Moreover, the sweepout $\{\Si_{t}\}_{t\in[-1, 1]}$ constructed there is just the flow of $\Si_{0}$ along $u_{1}\nu$, where $\nu$ is the unit normal vector fields of $\Si_{0}$. Suppose the index of $\Si_{0}$ is grater or equal to two, then we can find an $L^{2}$ orthonormal eigenbasis $\{v_{1}, v_{2}\}\subset C^{\infty}(\Si_{0})$ of $L_{\Si_{0}}$ with negative eigenvalues. A linear combination will give a $v_{3}\in C^{\infty}(\Si_{0})$, such that
\begin{equation}\label{orthogonality}
\int_{\Si_{0}}v_{3}L_{\Si_{0}}u_{1}d\mu=0,\quad v_{3}\neq 0.
\end{equation}
Let $\ti{X}=v_{3}\nu$ be another normal vector field, and extend it to a tubular neighborhood of $\Si_{0}$. Denote $\{\ti{F}_{s}\}_{s\in[-\ep, \ep]}$ to be the flow of $\ti{X}$, hence $\ti{F}_{s}$ are all isotopies. Now let $\Si_{s, t}=\ti{F}_{s}(\Si_{t})$, and consider the two parameter family of generalized smooth family $\{\Si_{s, t}\}_{(s, t)\in[-\ep, \ep]\times[-1, 1]}$. Notice that $\Si_{s, t}$ is then a smooth family for $(s, t)\in[-\ep, \ep]\times[-\ep, \ep]$ for $\ep$ small enough by $(c)$ in Proposition \ref{existence of good sweepout}. Denote $\ti{f}(s, t)=\mH^{n}(\Si_{s, t})$. Then $\nabla\ti{f}(0, 0)=0$ (by minimality of $\Si_{0}$), $\frac{\partial^{2}}{\partial t\partial s}\ti{f}(0, 0)=0$ (by (\ref{orthogonality})), and $\frac{\partial^{2}}{\partial t^{2}}f(0, 0)<0$, $\frac{\partial^{2}}{\partial s^{2}}f(0, 0)<0$ (by negativity of eigenvalues). So there exists $\de>0$ small enough, $\ti{f}(\de, t)<\ti{f}(0, 0)$ for all $t$, since $\ti{f}(0, t)<\ti{f}(0, 0)$ for all $t\neq 0$ by $(b)$ in Proposition \ref{existence of good sweepout}. By Remark \ref{remark about sweepout}, $\{\Si_{\de, t}\}_{t\in[-1, 1]}$ is a sweepout in the sense of Definition \ref{definition of sweepout}. By Proposition \ref{continuity and no mass concentration}, Theorem \ref{generating (1, M) homotopy sequence} and Theorem \ref{identification of (1, M) homotopy class}, we can construct a $(1, \M)$-homotopy sequence $\{\phi^{\de}_{i}\}_{i\in\N}$, such that $\{\phi^{\de}_{i}\}_{i\in\N}\in \Pi_{M}$, and
$$\bL\big(\{\phi^{\de}_{i}\}_{i\in\N}\big)\leq \sup_{t\in[-1, 1]}\ti{f}(\de, t)<\ti{f}(0, 0)=V(\Si_{0})=W_{M},$$
which is hence a contradiction to the fact that $\bL(\Pi_{M})=W_{M}$. So we proved Claim \ref{index one} and hence case $(i)$.
\end{proof}
\begin{remark}

We used the same idea to prove the index bound as in \cite{MN}\cite{MN2}. However they a prior need the existence of a least area embedded minimal surface among a family of embedded minimal surfaces, while in our case, the existence of a least area minimal hypersurface is just a by-product of the min-max construction and the existence of good sweepouts (Proposition \ref{existence of good sweepout}, Proposition \ref{existence of good sweepout2}).
\end{remark}


\section{Appendix}\label{appendix}

First we give the proof for Claim \ref{claim for existence of good sweepout} in Proposition \ref{existence of good sweepout}.

\begin{proof}(of Claim \ref{claim for existence of good sweepout} in Proposition \ref{existence of good sweepout})
Denote $U_{s_{0}}=F([-s_{0}, s_{0}]\times\Si)$ for $0<s_{0}\leq \ep$. It is easily to see that $\{\Si_{s}\}_{s\in[-\ep, \ep]}$ is a foliation corresponding to the level set of a function $f$ defined in a neighborhood $U_{\ep }$ of $\Si$, such that $f(\Si_{s})=s$. In fact, using coordinates $(s, x)\in [-\ep, \ep]\times\Si$ for $U_{\ep}=F([-\ep, \ep]\times\Si)$,
$$f(s, x)=s=\frac{d^{\pm}(s, x)}{u_{1}(x)}, \quad f\in C^{\infty}(U_{\ep}),$$
where $d^{\pm}: U_{\ep}\rightarrow \R$ is the signed distance function with respect to $\Si$, i.e.
\begin{displaymath}
\left. d^{\pm}(x)= \Big\{ \begin{array}{ll}
dist(x, \Si), \quad \textrm{ if $x\in M_{1}$};\\
-dist(x, \Si), \quad \textrm{ if $x\in M_{2}$}.
\end{array} \right. 
\end{displaymath}
Since $|\nabla d^{\pm}|=1$, $|f|\leq\ep$ and $\nabla f=\frac{\nabla d^{\pm}-f\nabla u_{1}}{u_{1}}$, we can choose $\ep$ small enough depending only on $u_{1}$, such that $|\nabla f|$ is bounded away from $0$ on $U_{\ep}$. Hence $f$ is a Morse function on $U_{\ep}$.

We want to cook up a Morse function $g$ on $M$, which coincides with $f$ on $U_{\frac{1}{2}\ep}$. First extend $f$ to be a smooth function on $M$ (denoted still by $f$), such that $f|_{M_{1, \frac{3}{4}\ep}}>\frac{3}{4}\ep$ (and $f|_{M_{2, \frac{3}{4}\ep}}<-\frac{3}{4}\ep$). Using the fact that the set of Morse function is dense in $C^{k}(M)$ for $k\geq 2$ (see \cite[Chap 6, Theorem 1.2]{H}), we can find a $C^{\infty}$ function $\ti{f}$, such that $\|f-\ti{f}\|_{C^{2}}$ is arbitrarily small. Choose a cutoff function $\varphi: M\rightarrow\R$, such that $\varphi\equiv 1$ on $U_{\frac{1}{2}\ep}$, and $\varphi\equiv 0$ outside $U_{\frac{3}{4}\ep}$. Let
$$g=\varphi f+(1-\varphi)\ti{f}=f+(1-\varphi)(\ti{f}-f).$$
Hence $g\equiv f$ in $U_{\frac{1}{2}\ep}$, and $g\equiv \ti{f}$ outside $U_{\frac{3}{4}\ep}$. In order to check that $g$ is a Morse function, we only need to check that in the middle region. Now
$$\nabla g=\nabla f+(1-\varphi)(\nabla\ti{f}-\nabla f)-\nabla\varphi(\ti{f}-f).$$
Since $|\nabla f|$ is bounded away from 0 on $U_{\ep}$, we can take $\|\ti{f}-f\|_{C^{2}}$ small enough to make sure that $|\nabla g|$ is bounded away from 0, hence $g$ is a Morse function.

Now take $\{\ti{\Si}_{s}\}$ to be the sweepout given by the level surface of $g$ (by Proposition \ref{level sets of morse function}). $\ti{\Si}_{s}=\Si_{s}$ since $g\equiv f$ in $U_{\frac{1}{2}\ep}$. $\ti{\Si}_{s}\subset M_{1, \frac{1}{2}\ep}$ (or $\subset M_{2, \frac{1}{2}\ep}$) when $s>\frac{1}{2}\ep$ (or $s<-\frac{1}{2}\ep$) follows from the fact that $g>\frac{1}{2}\ep$ on $M_{1, \frac{1}{2}\ep}$ (or $g<-\frac{1}{2}\ep$ on $M_{2, \frac{1}{2}\ep}$). A reparameterization gives the sweepout in the claim.
\end{proof}

Now we give the proof of Claim \ref{current in submanifold} in Proposition \ref{orientation and multiplicity}. The proof is elementary, but does not appear in standard reference, so we add it here for completeness.

\begin{proof}
(of Claim \ref{current in submanifold} in Proposition \ref{orientation and multiplicity}) Denote $\Si_{0}=\cup_{i=1}^{l}\Si_{i}$. First we show that $T_{0}$ is an integral current in $\Si_{0}$. Since $T_{0}$ is an integral current in $M$, it is represented as $T_{0}=\underline{\underline{\tau}}(N, \theta, \xi)$ (see \cite[\S 27.1]{S}), where $N$ is a countably $n$-rectifiable set, $\theta$ an integer-valued locally $\mH^{n}$ integrable function, and $\xi$ equals the orienting $n$-form of the approximated tangent plane $T_{x}N$ for $\mH^{n}$ a.e. $x\in N$. As $N$ lies in the support of $T_{0}$, hence in $\Si_{0}$, $T_{0}$ also represents an integral current in $\Si_{0}$, and we denote it as $T^{\pr}_{0}$.

Now let us show that $\partial T_{0}^{\pr}=0$ as current in $\Z_{n}(\Si_{0})$. We only need to show that for any compactly supported smooth $n-1$ form $\psi\in\La^{n-1}_{c}(\Si_{0})$, we have $\partial T^{\pr}_{0}(\psi)=0$. By using partition of unity, we can restrict to the case when $\psi$ is supported in a local coordinate chart.

Assume that the support of $\psi$ lies in $U\cap\Si_{0}$, where $U$ is a coordinates chart for $M$, with coordinates $\{x_{1}, \cdots, x_{n-1}, y\}$, and $U\cap\Si_{0}$ is given by $y=0$. We can easily extend $\psi$ smoothly to a neighborhood of $U\cap\Si_{0}$, denoting by $\ti{\psi}\in \La^{n-1}_{c}(U)$, such that $\mL_{\partial y}\ti{\psi}=0$ near $U\cap\Si_{0}$. In fact, this can be achieved by extending the coefficients of $\psi$ to $U$ trivially, so that those coefficients do not depend on $y$ near $U\cap\Si_{0}$. Hence $d\ti{\psi}|_{U\cap\Si_{0}}=d\psi$. So
$$\partial T^{\pr}_{0}(\psi)=T^{\pr}_{0}(d\psi)=T^{\pr}_{0}(d\ti{\psi}|_{U\cap\Si_{0}})=T_{0}(d\ti{\psi})=\partial T_{0}(\ti{\psi})=0,$$
where the third $``="$ follows from the integral formula (\cite[page 146]{S}) for integral currents. Writing $T^{\pr}_{0}$ as $T_{0}$ again, we finish the proof.
\end{proof}


\parindent 0ex
Department of Mathematics, Stanford University\\
Stanford, California 94305\\
E-mail: xzhou08@math.stanford.edu

\end{document}